\documentclass[10pt,a4paper]{article}
\usepackage[utf8]{inputenc}
\usepackage{amsmath}
\usepackage{cite}
\usepackage{pifont}
\usepackage{bm}
\usepackage[american]{babel}
\usepackage[all]{xy}
\usepackage{enumitem}
\usepackage{anysize}
\usepackage[dvipsnames]{xcolor}
\usepackage{graphicx}
\usepackage{stmaryrd}
\usepackage{amsfonts,amstext,amssymb,amsmath,amsthm,pspicture,bigints}

\usepackage{fullpage}
\usepackage{moreverb}
\usepackage{pifont}
\usepackage{pst-all}
\usepackage[left=2cm,right=2cm,top=2cm,bottom=2cm]{geometry}
\usepackage{esint}
\usepackage{fourier-orns}
\usepackage{mathrsfs}
\usepackage{array}
\newcommand{\T}{\mathbb{T}}
\usepackage{hyperref}
\usepackage{cleveref}
\newcommand{\ii }{{\rm i} }
\usepackage{authblk}
\newcolumntype{C}[1]{>{\centering\arraybackslash}b{#1}}
\newcolumntype{R}[1]{>{\raggedleft\arraybackslash}b{#1}}
\newcolumntype{L}[1]{>{\raggedright\arraybackslash}b{#1}}
\newcolumntype{M}[1]{>{\centering}m{#1}}
\newtheorem{theo}{Theorem}[section]

\newtheorem{lem}{Lemma}[section]
\newtheorem{prop}{Proposition}[section]

\newtheorem{remark}{Remark}[section]

\usepackage{bbm}
\numberwithin{equation}{section}
\newcommand{\R}{\mathbb{R}}
\newcommand{\x}{\mathbf{x}}
\newcommand{\C}{\mathbb{C}}

\newcommand{\pare}[1]{\left( #1 \right)}
\newcommand{\bra}[1]{\left[ #1 \right]}

\newcommand{\av}[1]{\left|#1\right|}
\newcommand{\bbra}[1]{\left\llbracket #1 \right\rrbracket}

\date{}
\title{Steady vortex sheets in presence of surface tension}
\author{Federico Murgante \hspace{1cm} Emeric Roulley \hspace{1cm} Stefano Scrobogna}
\begin{document}
	\maketitle
	\begin{abstract}
		We prove a bifurcation result of uniformly-rotating/stationary non-trivial vortex sheets near the circular distribution for a model of two irrotational fluids with same density taking into account surface tension effects. As bifurcation parameters, we play with either the speed of rotation, the surface tension coefficient or the mean vorticity.
	\end{abstract}
	\tableofcontents
	\section{Introduction}

	We consider a planar Euler system for two irrotational fluids with same density (constant equal to $1$) separated by an interface $\Gamma(t)$ homeomorphic to a circle and parametrized by $z(t,\cdot):\mathbb{T}\rightarrow\mathbb{R}^2.$ This interface divides the plane into two open components $\Omega^{\pm}(t)$ with $\Omega^{-}(t)$ bounded and $\Omega^{+}(t)$ unbounded. The evolutionary system is thus composed of the following equations
	\begin{equation}\label{Euler}
		\begin{cases}
			u_t^{\pm}+u^{\pm}\cdot\nabla u^{\pm}+\nabla p^{\pm}=0, & \textnormal{in }\Omega^{\pm}(t),\\
			\big(z_t-u^{\pm}|_{\Gamma(t)}\big)\cdot z_{x}^\perp=0,& \textnormal{at }\Gamma(t),\\
			(p^--p^+)|_{\Gamma(t)}=\sigma\mathcal{K}(z), & \textnormal{at }\Gamma(t),\\
			u^{+}(t,\mathbf{x}) \to 0, & \text{as } |\mathbf{x}|\to +\infty,\\
			\nabla\cdot u^{\pm}=0, & \textnormal{in }\Omega^{\pm}(t),\\
			\nabla^{\perp}\cdot u^{\pm}=0, & \textnormal{in }\Omega^{\pm}(t).
			
		\end{cases}
	\end{equation}
	In the above set of equations, the quantities $u^{\pm},p^{\pm}$ are respectively the velocity field and pressure inside the domain $\Omega^{\pm}.$ The parameter $\sigma\geqslant 0$ is the surface tension coefficient and $\mathcal{K}(z)$ is the curvature defined by
	$$\mathcal{K}(z)\triangleq
	-\frac{z_{x}^{\perp}\cdot z_{xx}}{|z_x|^3}\cdot$$
	The last equation in \eqref{Euler} implies that the vorticity distribution $\boldsymbol{\omega}$ is localized on the curve $\Gamma(t)$ at time $t,$ namely
	\begin{align}\label{eq:vorticty_Dirac}
		\boldsymbol{\omega}(t,\mathbf{x})=\omega(t,x)\delta\big(\mathbf{x}-z(t,x)\big),\qquad\mathbf{x}\in\mathbb{R}^2,\quad x\in\mathbb{T}.
	\end{align}
	Such a solution is called \textit{vortex sheet}.  The velocity fields are recovered through the so-called Biot-Savart law
	\begin{equation}
		\label{eq:upm_BS}
		u^{\pm}(t,\mathbf{x})=\int_{\mathbb{T}}\frac{\big(\mathbf{x}-z(t,y)\big)^{\perp}}{\big|\mathbf{x}-z(t,y)\big|^2}\omega(t,y)\textnormal{d} y,
	\end{equation}
	where we used the following convention for the integral on the torus $\T$
	$$\int_{\mathbb{T}}f(y)\textnormal{d} y\triangleq\frac{1}{2\pi}\int_{0}^{2\pi}f(y)\textnormal{d} y.$$
	The equations for the system \eqref{Euler} in the case in which the vorticty is a vortex sheet (i.e. \eqref{eq:vorticty_Dirac}) are called {\it Kelvin-Helmholtz system}. From \eqref{eq:upm_BS} and standard suppression of the pressure via Leray projectors it is clear that the bulk quantities $ u^\pm $ and $ p^\pm $ can be expressed in terms of the interface quantities $ \omega $ and $ z $, thus recasting \eqref{Euler} as a {\it Contour Dynamic Equation (CDE)} on $ \mathbb{T} $. The resulting system, derived in Appendix \ref{sec:contour_derivation}, writes as 
	\begin{equation}\label{VS system}
		\left\lbrace
		\begin{aligned}
			& \pare{z_t - \textnormal{BR}(z)\omega\big. }\cdot z_x^\perp =0,
			\\
			&  \omega_t = \pare{
				\omega \ \frac{\pare{z_t - \textnormal{BR}(z)\omega}\cdot z_x}{\av{z_x}^2} 
			}_x - \sigma \pare{\mathcal{K}\pare{z}}_x.
		\end{aligned}
		\right.
	\end{equation}
	where the {\it Birkhoff-Rott integral operator} is defined by
	$$\textnormal{BR}(z)\omega(t,x)\triangleq
	\textnormal{p.v.}\int_{\mathbb{T}}\frac{\big(z(t,x)-z(t,y)\big)^{\perp}}{|z(t,x)-z(t,y)|^2}\omega(t,y)\textnormal{d} y.$$
	Let us choose a parametrization in the form
	\begin{align}\label{graph}
		z(t,x)=R(t,x)e^{\ii x},\qquad R(t,x)\triangleq
		\sqrt{1+2\eta(t,x)}.
	\end{align}
	The choice of the parametrization \eqref{graph}, which is a graph on the unit circle, allows us to recast the system \eqref{VS system} in a more congenial form. The detailed computations are performed in Appendix \ref{appendix derive eq} and produce the system
	\begin{equation}\label{eq:KH2}
		\left\lbrace
		\begin{aligned}
			&\eta_t=-\frac{1}{2} \mathscr{H}(\eta)[\omega],\\
			&\omega_t=-\left(\frac{\omega}{2}\mathscr{D}_0(\eta)[\omega]\right)_x-\sigma\big(\mathscr{K}(\eta)\big)_x,
		\end{aligned} 
		\right. 
	\end{equation}
	where
	\begin{align}
		\mathscr{H}(\eta)[\omega]&\triangleq
		\eta_{x}\mathscr{D}_0(\eta)[\omega]+\mathscr{H}_0(\eta)[\omega],\label{def H}\\
		\mathscr{D}_0(\eta)[\omega](x)&\triangleq
		\textnormal{p.v.}\int_{\mathbb{T}}\frac{1-\sqrt{\frac{1+2\eta(y)}{1+2\eta(x)}}\cos(x-y)}{1+\eta(x)+\eta(y)-\sqrt{1+2\eta(x)}\sqrt{1+2\eta(y)}\cos(x-y)}\omega(y)\textnormal{d} y,\label{def D0}\\
		\mathscr{H}_0(\eta)[\omega](x)&\triangleq
		\textnormal{p.v.}\int_{\mathbb{T}}\frac{\sqrt{1+2\eta(x)}\sqrt{1+2\eta(y)}\sin(x-y)}{1+\eta(x)+\eta(y)-\sqrt{1+2\eta(x)}\sqrt{1+2\eta(y)}\cos(x-y)}\omega(y)\textnormal{d} y,\label{def H0}\\
		\mathscr{K}(\eta)&\triangleq
		\frac{\eta_{xx}-2\left(\frac{\eta_{x}}{R}\right)^2}{\left(R^2+\left(\frac{\eta_{x}}{R}\right)^2\right)^{\frac{3}{2}}}-\left(R^2+\left(\frac{\eta_{x}}{R}\right)^2\right)^{-\frac{1}{2}}.\label{def K}
	\end{align} 
	Using the divergence free property of the flow, Stokes' Theorem, the second equation in \eqref{Euler} and the first equation in  \eqref{basic idtt2},  we get 
	$$0=\int_{\Omega^-(t)}\nabla\cdot u^-(t,\x)\,d\x=\int_{0}^{2\pi}u^{-}\big(t,z(t,x)\big)\cdot z_x^\perp(t,x)\, d x =\int_{0}^{2\pi} z_t(t,x)\cdot z_x^\perp(t,x)\,dx =-\frac{d}{dt}\int_{0}^{2\pi}\eta(t,x)\,dx.$$
	Therefore, the space average of $\eta$ is preserved and we impose
	$$\int_{\mathbb{T}}\eta(x)dx=0.$$
	%
	%
	%
	%
	The second equation in \eqref{eq:KH2} implies the conservation of the mean vorticity and in the sequel, we denote
	$$\gamma\triangleq\int_{\mathbb{T}}\omega(x)dx.$$
	Let us introduce the velocity potential $\psi$ through the relation
	\begin{equation}\label{def:psi}
		\omega=\gamma+\psi_x,\qquad\textnormal{i.e.}\qquad\psi\triangleq\partial_{x}^{-1}(\omega-\gamma).
	\end{equation}
	In the new variables $(\eta,\psi)$, the system \eqref{eq:KH2} becomes
	\begin{equation}\label{eq:KH3}
		\left\lbrace
		\begin{aligned}
			&\eta_t=-\frac{1}{2} \mathscr{H}(\eta)[\psi_x]-\frac{\gamma}{2} \mathscr{H}(\eta)[1],\\
			&\psi_t=-\frac{\psi_x+\gamma}{2}\mathscr{D}_0(\eta)[\psi_x+\gamma]-\sigma\mathscr{K}(\eta).
		\end{aligned} 
		\right. 
	\end{equation}
	Let us emphazise that to obtain the second equation in \eqref{eq:KH3}, we have integrated in space the second equation in \eqref{eq:KH2}. Therefore, the second equation is well-defined up to a time-dependent additive constant.\\
	
	The system \eqref{eq:KH3} can be reformulated in a Hamiltonian form, akin to the approach taken by \cite{BB97} for perturbations of the flat interface. The Hamiltonian structure is an important ingredient in the study of small divisors problems for PDEs to construct for instance quasi-periodic solutions \cite{BBHM18,BFM21,BFM24,BHM22,BM20,FG20a,FG20b,GIP23,HR21,HR22,HHR23} and a fundamental tool to study long time stability in presence of resonances \cite{BMM2022,BMM2023,BD2016,MMS2024, BCGS2023}. Since it falls outside the scope of this manuscript and does not play an important role in our bifurcation analysis, we omit the derivation of the Hamiltonian formulation and we leave it to our future investigations.
	\subsection{Main results}
	We look for uniformly rotating vortex sheets, that is traveling wave solutions of the system \eqref{eq:KH3}, namely
	$$(\eta,\psi)(t,x)=(\check{\eta},\check{\psi})(x-ct),\qquad c\in\mathbb{R},\qquad\check{\eta},\check{\psi}\in L^2(\mathbb{T}).$$
	When $c=0$, the corresponding solutions are stationary. The equations \eqref{eq:KH3} rewrite in this context (for simplicity of notations, we still denote $(\eta,\psi)$ instead of $(\check{\eta},\check{\psi})$)
	\begin{equation}\label{eq:KH4}
		\begin{cases}
			c\,\eta_x+\frac{1}{2} \mathscr{H}(\eta)[\psi_x]+\frac{\gamma}{2} \mathscr{H}(\eta)[1]=0,\\
			c\,\psi_x+\frac{\psi_x+\gamma}{2}\mathscr{D}_0(\eta)[\psi_x+\gamma]+\sigma\mathscr{K}(\eta)=0.
		\end{cases} 
	\end{equation}
	Given $\mathbf{m}\in\mathbb{N}^*$, we say that a uniformly rotating vortex sheet is $\mathbf{m}$-fold if the functions $\check{\eta}$ and $\check{\psi}$ are $\frac{2\pi}{\mathbf{m}}$ periodic. 
	\\
	
	The couple $(\eta,\psi)=(0,0)$ is a trivial solution of \eqref{eq:KH3} for any values of the parameters $c,\sigma,\gamma$. This solution corresponds to the family of circular stationary vortex sheets given by
	$$z(x)=e^{\ii x},\qquad\omega\equiv\gamma,\qquad u^{-}(\mathbf{x})\equiv 0,\qquad u^{+}(\mathbf{x})=\gamma  \frac{\mathbf{x}^\perp}{|\mathbf{x}|^2}\cdot$$
	We refer the reader to Lemma \ref{lemma trivial sol} for more details. Our main result, stated below, gives the existence of non-trivial solutions which are small amplitude perturbations of this stationary state. 
	\begin{theo}
		\label{thm bif vortex sheet}\textbf{(Local bifurcation of vortex sheets from the circular distribution)}
		\begin{enumerate}[label=(\roman*)]
			\item Let $\sigma>0$, $\gamma\in\mathbb{R}$ and $\mathbf{m}\in\mathbb{N}^*$. Assume that 
			\begin{equation}\label{cond c thm}
				\frac{\gamma^2-\sigma}{\sigma\mathbf{m}^2}\not\in\mathbb{N}^*
			\end{equation}
			supplemented by one of the following conditions
			\begin{enumerate}
				\item $4\sigma(2-\sqrt{3})<\gamma^2<4\sigma(2+\sqrt{3}).$
				\item $\gamma^2\in[0,\infty)\setminus\left[4\sigma(2-\sqrt{3}),4\sigma(2+\sqrt{3})\right]$ and $\mathbf{m}\in\mathbb{R}\setminus[m_-(\sigma,\gamma),m_+(\sigma,\gamma)],$ with
				$$m_{\pm}(\sigma,\gamma)=\frac{\gamma^2}{4\sigma}\pm\frac{1}{4\sigma}\sqrt{(\gamma^2-8\sigma)^2-48\sigma^2}\cdot$$
			\end{enumerate}
			Then, there exist two branches of $\mathbf{m}$-fold uniformly rotating vortex-sheets bifurcating from the circular distribution at angular speed
			$$c_{\mathbf{m}}^{\pm}(\sigma,\gamma)=-\frac{\gamma}{2}\pm\frac{1}{2}\sqrt{2\sigma\mathbf{m}-\gamma^2+\frac{2(\gamma^2-\sigma)}{\mathbf{m}}}\cdot$$
			\item Let $(c,\gamma)\in\mathbb{R}^2\setminus\{(0,0)\}.$ Fix $\mathbf{m}\in\mathbb{N}\setminus\{0,1\}$ such that
			$$\mathbf{m}>\frac{2\gamma^2}{\left(2c+\gamma\right)^2+\gamma^2}\cdot$$
			Assume in addition that
			\begin{equation}\label{condition sigma thm}
				\frac{(2\mathbf{m}-1)\gamma^2-(2c+\gamma)^2}{\mathbf{m}(2c+\gamma)^2+(\mathbf{m}-2)\gamma^2}\not\in\mathbb{N}^*.
			\end{equation}
			Then, there exists one branch of $\mathbf{m}$-fold uniformly rotating vortex-sheets with speed $c$ bifurcating from the circular distribution for
			$$\sigma_{\mathbf{m}}(c,\gamma)=\frac{\mathbf{m}\left(2c+\gamma\right)^2+(\mathbf{m}-2)\gamma^2}{2(\mathbf{m}^2-1)}\cdot$$
			\item Let $\sigma>0$ and $\mathbf{m}\in\mathbb{N}\setminus\{0,1\}.$ Then, there exist two branches of $\mathbf{m}$-fold stationary vortex-sheets bifurcating from the circular distribution for
			$$\gamma_{\mathbf{m}}^{\pm}(\sigma)=\pm\sqrt{\sigma(\mathbf{m}+1)}\cdot$$
		\end{enumerate}
	\end{theo}
	\begin{remark} Let us make some remarks.
		\begin{enumerate}
			\item The condition \eqref{cond c thm} is satisfied in particular for $$\frac{\gamma^2}{\sigma}\in\mathbb{R}\setminus\mathbb{Q}.$$
			\item Notice that formally $\sigma_{\mathbf{m}}(0,\gamma)=\frac{\gamma^2}{\mathbf{m}+1},$ which is in accordance with the point $(iii)$. Nevertheless, the condition  \eqref{condition sigma thm} is not compatible with $c=0$ but is satisfied in particular for $$\left(\frac{2c}{\gamma}+1\right)^2\in\mathbb{R}\setminus\mathbb{Q}.$$
		\end{enumerate}
	\end{remark}
	The Theorem \ref{thm bif vortex sheet} is obtained by means of local bifurcation theory, more precisely by applying Crandall-Rabinowitz Theorem \ref{thm CR}. For this aim, we see solutions to \eqref{eq:KH4} as zeros of a nonlinear functional $\mathcal{F}$ (defined by \eqref{def F1}-\eqref{def F2}):
	$$\mathcal{F}(c,\sigma,\gamma,\eta,\psi)=0.$$
	We prove in Lemma \ref{lemma trivial sol} that $(\eta,\psi)=(0,0)$ is indeed a solution. Then, in Proposition \ref{prop regularity}, we study the regularity of $\mathcal{F}$ and we give in Proposition \ref{prop lin op} the expression of its differential $\mathcal{L}_{c,\sigma,\gamma}$ at the trivial solution $(\eta,\psi)=(0,0).$ We dispose of three parameters to bifurcate: the speed of rotation $c$, the surface tension coefficient $\sigma$ and the mean vorticity $\gamma.$ The corresponding analysis are led in Subsections \ref{subsec c}, \ref{subsec sigma} and \ref{subsec gamma}, respectively. Let us emphasize that the restrictions \eqref{cond c thm} and \eqref{condition sigma thm} are made for avoiding spectral collisions and get a one dimensional kernel for $\mathcal{L}_{c,\sigma,\gamma}.$ The transversality conditions are obtained via a duality argument as developed in \cite{GHR2023,HHRZ24,Roulley2023a,Roulley2023c}.
	
	\subsection{Known results}
	
	The study of the motion of two immiscible fluids separated by a sharp interface has been a subject of study since the mid-19th century, when von Helmholtz derived the evolution equations for such phenomenon in 1858 in \cite{Helmholtz1858}. When there is no surface tension and no gravity the problem is generally ill-posed in Sobolev spaces, even at linear level (cf. \cite[Section 9.3]{MB02}), this phenomenon, known as {\it Kelvin-Helmholtz instability}, lead to the formation of vortices or waves at the interface or within the shear layer. It is named after Lord Kelvin (William Thomson) and Hermann von Helmholtz, who independently investigated the instability in the late 19th century.\\
	
	When the fluid interface is a perturbation of the flat rest state and there are no surface tension effects ($ \sigma=0 $) stability in the analytic framework was proved in \cite{SS85,SSBF81,HH03}. Adding the effects of the surface tension to \eqref{Euler} ($\sigma>0$) induces stabilizing effects and the resulting equations are well-posed, linearly \cite{HLS97}, and nonlinearly \cite{Ambrose03,AM07,Ambrose07,CCS08,SZ11,CCG2012}, even in presence of two fluids with different densities. A more comprehensive stability criterion has been investigated by D. Lannes in \cite{Lannes13} in the context of two fluids with nonzero density, when gravitational effects are taken in consideration. We highlight that when we have a shear flow the presence of the surface tension is crucial in order to avoid the growth of high-modes, that is characteristically induced by the Kelvin-Helmholtz instability, otherwise the equations are ill-posed \cite{Wu06,L02,CO1989}, in sharp contrast with what happens in the context of the Water-Waves equations. We mention as well the foundational work of Delort \cite{Delort1991} which proved the existence of global weak solutions for measured-valued vorticities that include the vortex sheet problem. \\
	
	The stability results discussed are constrained by an existence time proportional to the capillarity coefficient $\sigma$, leaving the long-term stability of \eqref{eq:KH3} largely unexplored. When $ \sigma =0 $, circles and lines are non-trivial steady solutions, as are segments of length $2a$ with strength $\omega(x) = \Omega\sqrt{a^2-x^2}$, rotating uniformly at angular velocity $\Omega$. This category also includes the family of solutions identified by B. Protas and T. Sakajo \cite{PS20} and the solutions \cite{CQZ2023a,CQZ2023b} of Cao-Qin-Zou. When $\Gamma(t)$ is a closed curve and $\sigma = 0$, recent findings \cite{GSPSY22} have demonstrated that steady solutions bifurcate from degenerate eigenvalues. In \cite{GSPSY21}, the same authors proved rigidity results, namely the non-existence of non-trivial solutions for a certain range of angular speed. Over the past decades, bifurcation techniques have been succesfully implemented to obtain traveling periodic solutions for instance, for the 3D liquid drop model with capillarity \cite{BJLM24} or in the vortex patch class where a broad spectrum of steady solutions for CDEs were discovered, starting with Burbea's seminal work \cite{Burbea82} and extended in \cite{HM13,HMW20,GHM2024,CCG2016,CCG2019,HMV15,HM16b,HM16a,DHHM16,HHMV16}. The previous list is non-exhaustive since restricted to the Eulerian model. Let us mention the recent generalization to a large class of models \cite{HXX2023}. This manuscript adds to this literature by presenting an existence result for the Kelvin-Helmholtz equations near a circular vortex with surface tension effects. The combined contributions of our study and \cite{GSPSY22} establish, for the first time, the presence of closed curves as global solutions to Kelvin-Helmholtz equations.\\
	
	\noindent\textbf{Ackowledgments:} Federico Murgante is supported by the ERC STARTING GRANT 2021 "Hamiltonian Dynamics, Normal Forms and Water Waves" (HamDyWWa), Project Number: 101039762. Emeric Roulley is supported by PRIN 2020 "Hamiltonian and Dispersive PDEs" project number: 2020XB3EFL. Stefano Scrobogna is supported by PRIN 2022 "Turbulent effects vs Stability in Equations from Oceanography" (TESEO), project number: 2022HSSYPN.
	
	\section{Proof of the results}
	This section is devoted to the proof of Theorem \ref{thm bif vortex sheet}. We first introduce the functional of interest and give the expression of its differential at the circular distribution. Then, we study the bifurcations with respect to the various parameters of the problem.
	\subsection{Functional of interest and its linear operator}
	Let us introduce the functional
	$$\mathcal{F}=(\mathcal{F}_1,\mathcal{F}_2),$$
	where 
	\begin{align}
		\mathcal{F}_1(c,\sigma,\gamma,\eta,\psi)&\triangleq
		c\,\eta_x+\frac{1}{2} \mathscr{H}(\eta)[\psi_x]+\frac{\gamma}{2} \mathscr{H}(\eta)[1],\label{def F1}\\
		\mathcal{F}_2(c,\sigma,\gamma,\eta,\psi)&\triangleq
		c\,\psi_x+\frac{\psi_x+\gamma}{2}\mathscr{D}_0(\eta)[\psi_x+\gamma]+\sigma\mathscr{K}(\eta).\label{def F2}
	\end{align}
	Then, clearly solutions of \eqref{eq:KH4} are zeros of the functional $\mathcal{F}$ and we have the following result stating that the circular distribution is indeed a trivial solution for any values of the parameters.
	\begin{lem}\label{lemma trivial sol}
		The state $(\eta,\psi)=(0,0)$ is a trivial solution. More precisely,
		$$\forall(c,\sigma,\gamma)\in\mathbb{R}\times(0,\infty)\times\mathbb{R},\quad\mathcal{F}(c,\sigma,\gamma,0,0)=0.$$
		This corresponds to
		\begin{equation}\label{stat:v_form}
			\omega\equiv\gamma,\qquad\Omega^{-}(t)=D(0,1)\qquad\textnormal{and}\qquad\begin{cases}
				u^{-}(\mathbf{x})=0, & \textnormal{if }\mathbf{x}\in D(0,1),\\
				u^{+}(\mathbf{x})=\gamma  \frac{\mathbf{x}^\perp}{|\mathbf{x}|^2}, & \textnormal{if }\mathbf{x}\in\mathbb{R}^2\setminus\overline{D(0,1)}.
		\end{cases}\end{equation}
	\end{lem}
	\begin{proof}
		From \eqref{def D0} and \eqref{def K}, one readily gets
		\begin{equation}\label{D00 K0}
			\mathscr{D}_0(0)[f](x)=\int_{\mathbb{T}}f(y)\textnormal{d} y,\qquad\mathscr{K}(0)=1.
		\end{equation}
		Then, using the classical relations
		\begin{equation}\label{trigo}
			\sin(u)=2\sin\left(\tfrac{u}{2}\right)\cos\left(\tfrac{u}{2}\right),\qquad1-\cos(u)=2\sin^2\left(\tfrac{u}{2}\right),
		\end{equation}
		we obtain from \eqref{def H} and \eqref{def H0}
		\begin{align}\label{H0}
			\mathscr{H}(0)[f](x)=\mathscr{H}_0(0)[f](x)=\textnormal{p.v.}\int_{\mathbb{T}}\frac{\sin(x-y)}{1-\cos(x-y)}f(y)\textnormal{d} y=\textnormal{p.v.}\int_{\mathbb{T}}\cot\left(\frac{x-y}{2}\right)f(y)\textnormal{d} y=\mathcal{H}[f](x),
		\end{align}
		where $\mathcal{H}$ denotes the $2\pi$-periodic Hilbert transform. Inserting the foregoing calculations into \eqref{def F1}-\eqref{def F2} gives
		$$\mathcal{F}_1(c,\sigma,\gamma,0,0)=\frac{\gamma}{2}\mathscr{H}(0)[1]=\frac{\gamma}{2}\mathcal{H}[1]=0$$
		and
		$$\mathcal{F}_2(c,\sigma,\gamma,0,0)=\frac{\gamma^2}{2}\mathscr{D}_0(0)[1]+\sigma\mathscr{K}(0)=\frac{\gamma^2}{2}+\sigma\sim0.$$
		Recall that the second equation is well-defined modulo constants. 
		Now let us turn to the fluid description. We employ the usual identification $\R^2\simeq \C$. For $(\eta,\psi)=0$, \eqref{graph} and \eqref{eq:upm_BS} imply that
		$$\omega\equiv\gamma\qquad\textnormal{and}\qquad z(x)=e^{\ii x}.$$
		Note that 
		\begin{equation} \label{eq:zx}
			\frac{z_x(x)}{\ii z(x)}\equiv1.
		\end{equation}
		We compute the complex conjugate of the velocity field for $\mathbf{x}\in \C\setminus\partial D(0,1)$ which, in view of \eqref{def:psi}, is given by the following holomorphic function
		\begin{align*}
			\overline{u^{\pm}(\mathbf{x})}&=\frac{\gamma}{2\pi\ii}\int_{0}^{2\pi}\frac{1}{\mathbf{x}-z(x)}dx\stackrel{\eqref{eq:zx}}{=}\frac{\gamma}{\ii}\frac{1}{2\pi\ii}\int_{0}^{2\pi}\frac{z_x(x)}{(\mathbf{x}-z(x))z(x)}dx=\frac{\gamma}{\ii}\frac{1}{2\pi\ii}\int_{\partial D(0,1)}\frac{dz}{(\mathbf{x}-z)z}=\begin{cases}
				0,&\mathbf{x}\in D(0,1),\\
				\frac{\gamma}{\ii\,\mathbf{x}}, & \mathbf{x}\not\in D(0,1),
			\end{cases}
		\end{align*}
		where, to obtain the last equality, we used the residue Theorem. Taking the complex conjugate, we get 
		$$u^{-}(\mathbf{x})=0,\qquad u^{+}(\mathbf{x})=\gamma\frac{\ii}{\overline{\mathbf{x}}}=\gamma\frac{\ii\mathbf{x}}{|\mathbf{x}|^2}=\gamma\frac{\mathbf{x}^\perp}{|\mathbf{x}|^2}\cdot$$
		This ends the proof of Lemma \ref{lemma trivial sol}. 
	\end{proof}
	We consider the following $\mathbf{m}$-fold ($\mathbf{m}\in\mathbb{N}^*$) Sobolev spaces with regularity index $s\geqslant0,$
	\begin{align*}
		H_{\textnormal{\tiny{even}},\mathbf{m}}^{s}&\triangleq
		\left\lbrace f(x)=\sum_{n=1}^{\infty}a_n\cos(n\mathbf{m}x),\quad\sum_{n=1}^{\infty}\langle n\rangle^{2s}|a_n|^2<\infty,\quad a_n\in\mathbb{R}\right\rbrace,\\
		H_{\textnormal{\tiny{odd}},\mathbf{m}}^{s}&\triangleq
		\left\lbrace f(x)=\sum_{n=1}^{\infty}b_n\sin(n\mathbf{m}x),\quad\sum_{n=1}^{\infty}\langle n\rangle^{2s}|b_n|^2<\infty,\quad b_n\in\mathbb{R}\right\rbrace,
	\end{align*}
	where we have used the classical notation $\langle n\rangle\triangleq\max(1,n).$ Then we set
	\begin{align*}
		X_{\mathbf{m}}^s\triangleq H_{\textnormal{\tiny{even}},\mathbf{m}}^{s+\frac{1}{4}}\times H_{\textnormal{\tiny{odd}},\mathbf{m}}^{s-\frac{1}{4}},\qquad Y_{\mathbf{m}}^{s}\triangleq
		H_{\textnormal{\tiny{odd}},\mathbf{m}}^{s-\frac{5}{4}}\times H_{\textnormal{\tiny{even}},\mathbf{m}}^{s-\frac{7}{4}}
	\end{align*}
	and for $r>0$,
	$$\mathcal{B}_{\mathbf{m}}^{s}(r)\triangleq
	\Big\{(\eta,\psi)\in X_{\mathbf{m}}^{s}\quad\textnormal{s.t.}\quad\|\eta\|_{H^{s+\frac{1}{4}}}+\|\psi\|_{H^{s-\frac{1}{4}}}<r\Big\}.$$
	In the next proposition, we state the regularity of the functional $\mathcal{F}$ with respect to these function spaces.
	\begin{prop}\label{prop regularity}
		Let $\mathbf{m}\in\mathbb{N}^*$ there exists $s_0>0$ such that for any $s\geqslant s_0$ there exists $r=r\pare{s}>0$ such that the functional 
		$$\mathcal{F}:\mathbb{R}\times(0,\infty)\times\mathbb{R}\times \mathcal{B}_{\mathbf{m}}^s(r)\longrightarrow Y_{\mathbf{m}}^{s}$$ 
		defined in \eqref{def F1}-\eqref{def F2}
		is well defined and of class $C^1.$
	\end{prop}
	\begin{proof}
		Assume $\eta$ and $\psi$ respectively even and odd (in space). Then $\eta_x$ and $\psi_{x}$ are respectively odd and even. Now by performing a change of variables $y\mapsto-y$ we can easily see from \eqref{def D0} and \eqref{def H0} that 
		$$f\textnormal{ even}\quad\Rightarrow\quad\Big(\mathscr{D}_0(\eta)[f]\textnormal{ even}\quad\textnormal{and}\quad\mathscr{H}_0(\eta)[f]\textnormal{ odd}\Big).$$
		Coming back to the expression \eqref{def H}, we see that
		$$f\textnormal{ even}\quad\Rightarrow\quad\mathscr{H}(\eta)[f]\textnormal{ odd}.$$
		With this in hand, we deduce from \eqref{def F1}-\eqref{def F2} that
		$$\mathcal{F}_1(c,\sigma,\gamma,\eta,\psi)\textnormal{ is odd}\qquad\textnormal{and}\qquad\mathcal{F}_2(c,\sigma,\gamma,\eta,\psi)\textnormal{ is even.}$$
		The $\mathbf{m}$-fold preserving property follows similarly by using the change of variables $y\mapsto y+\tfrac{2\pi}{\mathbf{m}}.$\\
		
		Since the parity and $ {\bf m} $-fold symmetry is preserved as explained above, using \eqref{def H}, \eqref{def F1} and \eqref{def F2} the result follows if we can prove that given $(\eta,\psi)\in\mathcal{B}_{\bf m}^s(r) $ then $\mathscr{H}_0(\eta)[\psi_x], \mathscr{D}_0 (\eta)[\psi_x] $ are in $ H^{s-\frac{5}{4}} $, the $ C^1 $ differentiability shall stem since the integrand functions that define the operators $ \mathscr{H}_0 $ and $ \mathscr{D}_0 $ are analytic in $ \eta $ and $ \psi $. We prove the statement for $ \mathscr{H}_0(\eta)[\psi_x] $ being the  one for $ \mathscr{D}_0(\eta)[\psi_x] $ similar. Notice that
		\begin{equation*}
			\mathscr{H}_0\pare{\eta}\bra{\psi_x} = |D| \psi + 2\, \text{p.v.} \int_{\mathbb{T}} \pare{ \mathsf{K} _z \pare{\frac{\Delta_z\eta}{1+2\eta}} -  1 } \frac{\psi_x \pare{x-z}}{2\tan\pare{z/2}} \text{d} z, 
		\end{equation*}
		where
		$$\Delta_{z}\eta\triangleq \frac{\eta(x)-\eta(x-z)}{2\sin\pare{z/2}}$$
		and
		\begin{equation}\label{eq:Kz}
			\mathsf{K}_z \pare{\mathsf{X}} 
			\triangleq 
			\frac{\sqrt{1-4\mathsf{X} \ \sin\pare{z/2} }}{2\pare{1-2\mathsf{X} \ \sin\pare{z/2}-\sqrt{1-4\mathsf{X} \ \sin\pare{z/2}}\cos z}} \ \pare{ 2\sin\pare{z/2} }^2 \ . 
		\end{equation}
		Notice that  $ \pare{z, \mathsf{X}}\mapsto \mathsf{K}_z\pare{\mathsf{X}} $ is analytic in $ (-\pi,\pi ) \times \pare{-\frac{1}{4}, \frac{1}{4}} $.
		In particular we can Taylor-expand in $ z $ the application $  z\mapsto\mathsf{K} _z \pare{\frac{\Delta_z\eta}{1+2\eta}} -  1  $ and we obtain that
		\begin{equation*}
			\mathsf{K} _z \pare{\frac{\Delta_z\eta}{1+2\eta}} -  1 =\mathsf{K}^0\pare{\frac{\eta_x}{1+2\eta} } + R^1\pare{\eta; x, z}, 
		\end{equation*}
		where
		\begin{align}
			\mathsf{K}^0\pare{\mathsf{X}}  & \triangleq
			- \frac{\mathsf{X}^2}{1 + \mathsf{X}^2}, \label{eq:k0}
			\\
			R^1\pare{\eta; x, z} & \triangleq
			\pare{ \mathsf{K} _z \pare{\frac{\Delta_z\eta}{1+2\eta}} -  1 } -\mathsf{K}^0\pare{\frac{\eta_x}{1+2\eta} }. \nonumber
		\end{align}
		We thus have that
		\begin{equation}\label{H0 remainder}
			\mathscr{H}_0\pare{\eta}\bra{\psi_x} = \pare{1+\mathsf{K}^0\pare{\frac{\eta_x}{1+2\eta}} } |D| \psi
			+
			\underbrace{2\, \text{p.v.} \int_{\mathbb{T}} R^1\pare{\eta; x, z}  \frac{\psi_x \pare{x-z}}{2\tan\pare{z/2}} \text{d} z}_{\triangleq
				R_{\mathscr{H}_0}\pare{\eta}\bra{\psi_x}}.  
		\end{equation}
		Applying standard Moser tame estimates and composition theorems to \eqref{eq:k0} it is immediate that there exists $ s_0 > 0 $ so that for any $ s\geqslant s_0 $ there exists a $ r=r\pare{s} > 0 $ so that
		\begin{equation}\label{C1 quasilinear H0}
			\pare{\eta, \psi}\mapsto  \pare{1+\mathsf{K}^0\pare{\frac{\eta_x}{1+2\eta}} } |D| \psi \in C^1\pare{\mathcal{B}_{\bf m}^s(r); H^{s-\frac{5}{4}}}, 
		\end{equation}
		thus we can focus our attention on the remainder term $ R_{\mathscr{H}_0}\pare{\eta}\bra{\psi_x} $ in \eqref{H0 remainder}. We have that, if $ s > 7/4 $
		\begin{equation}\label{RH0bound1}
			\left\|R_{\mathscr{H}_0}\pare{\eta}\bra{\psi_x}\right\| _{H^{s-\frac{5}{4}}}\lesssim \left\| \frac{R^1\pare{\eta;x,z}}{2\tan\pare{z/2}}\right\|_{L^\infty_z H^{s-\frac{5}{4}}_x} \left\|\psi\right\| _{H^{s-\frac{1}{4}}}
		\end{equation}
		Since $ R^1\pare{\eta;x, z} $ is a Taylor-1 remainder its explicit expression is given by
		\begin{equation}\label{eq:R1_explicit}
			R^1\pare{\eta;x,z}=z\int_0^1 \pare{ \partial_z \mathsf{K}_{\vartheta z}\pare{\frac{\Delta_{\vartheta z}\eta}{1+2\eta}} + \mathsf{K}' _{\vartheta z}\pare{\frac{\Delta_{\vartheta z}\eta}{1+2\eta}} \eta'\pare{x-\vartheta z} }d\vartheta, 
		\end{equation}
		so that from \eqref{eq:R1_explicit} and the fact that $ \partial_z\mathsf{K}_z\pare{0}=0 $ it is clear that if $ \|\eta\|_{H^{s-\frac{1}{4}}} \ll 1 $ then
		\begin{equation}\label{RH0bound2}
			\sup_{z\in\mathbb{T}}	\left\| \frac{R^1\pare{\eta;\bullet,z}}{2\tan\pare{z/2}}\right\|_{ H^{s-\frac{5}{4}}}\lesssim\|\eta\|_{H^{s-\frac{1}{4}}}, 
		\end{equation}
		so that \eqref{RH0bound2} and \eqref{RH0bound1} prove that (after a relabeling of $s_0,s$ and $r$, if needed)
		\begin{equation}\label{C0 RH0}
			\pare{\eta, \psi}\mapsto  R_{\mathscr{H}_0}\pare{\eta}\bra{\psi_x}  \in C^0 \pare{\mathcal{B}_{\bf m}^s(r); H^{s-\frac{5}{4}}}.
		\end{equation}
		The fact that the application (after a relabeling of $ s_0,s $ and $ r $, if needed)
		\begin{equation}
			\label{C1 RH0}
			\pare{\eta, \psi}\mapsto  R_{\mathscr{H}_0}\pare{\eta}\bra{\psi_x}  \in C^1\pare{\mathcal{B}_{\bf m}^s(r); H^{s-\frac{5}{4}}}.
		\end{equation}
		can be proved by computations similar to the ones performed to prove \eqref{C0 RH0}, and are omitted for the sake of brevity. Combining \eqref{C1 quasilinear H0}, \eqref{C1 RH0} and \eqref{H0 remainder}, we obtain that (after a relabeling of $ s_0,s $ and $ r $, if needed)
		\begin{equation}
			\label{C1 H0}
			\pare{\eta, \psi}\mapsto \mathscr{H}_0\pare{\eta}\bra{\psi_x}  \in C^1\pare{\mathcal{B}_{\bf m}^s(r); H^{s-\frac{5}{4}}}.
		\end{equation}
		Computations similar to the ones used to prove \eqref{C1 H0} show us that
		\begin{equation}
			\label{C1 D0}
			\pare{\eta, \psi}\mapsto \mathscr{D}_0\pare{\eta}\bra{\psi_x}  \in C^1\pare{\mathcal{B}_{\bf m}^s(r); H^{s-\frac{5}{4}}}.
		\end{equation}
		Putting together \eqref{C1 D0} and \eqref{C1 H0} proves, combined with standard product estimates, that, fixed $c,\gamma$ and $\sigma$
		\begin{equation}\label{reg partielle}
			\pare{\eta, \psi} \mapsto \big(\mathcal{F}_1(c,\sigma,\gamma,\eta,\psi) , \mathcal{F}_2(c,\sigma,\gamma,\eta,\psi) - \sigma\mathscr{K}(\eta)\big) \in C^1\pare{\mathcal{B}^s_{{\bf m}}\pare{r};  H^{s-\frac{5}{4}}_{\text{odd}, {\bf m}} \times  H^{s-\frac{5}{4}}_{\text{even}, {\bf m}} }.
		\end{equation}
		In addition, it is clear from \eqref{def K} and composition estimates that for $\|\eta\|_{H^{s+\frac{1}{4}}}\ll1,$ 
		\begin{equation}\label{reg scrK}
			\|\mathscr{K}(\eta)\|_{H^{s-\frac{7}{4}}}<\infty.
		\end{equation}
		From \eqref{reg partielle} and \eqref{reg scrK}, we deduce that
		$$\pare{\eta, \psi} \mapsto \big(\mathcal{F}_1(c,\sigma,\gamma,\eta,\psi) , \mathcal{F}_2(c,\sigma,\gamma,\eta,\psi)\big) \in C^1\pare{\mathcal{B}^s_{{\bf m}}\pare{r};  H^{s-\frac{5}{4}}_{\text{odd}, {\bf m}} \times  H^{s-\frac{7}{4}}_{\text{even}, {\bf m}} }.$$
		Thus, since the differentiability in  $c,\gamma$ and $\sigma$ is immediate from \eqref{def F1}-\eqref{def F2}, we conclude the desired result. 
	\end{proof}
	
	Our next goal is to linearize the system \eqref{eq:KH4} at the trivial solution $(\eta,\psi)=(0,0).$ The corresponding operator has a good Fourier multiplier structure and enjoys Fredholmness property with respect to the function spaces introduced above. More precisely, we have the following proposition.
	\begin{prop}\label{prop lin op}
		Let $(c,\sigma,\gamma)\in\mathbb{R}\times(0,\infty)\times\mathbb{R}.$ We denote $$\mathcal{L}_{c,\sigma,\gamma}\triangleq
		d_{\eta,\psi}\mathcal{F}(c,\sigma,\gamma,0,0).$$
		\begin{enumerate}[label=(\roman*)]
			\item The operator $\mathcal{L}_{c,\sigma,\gamma}$ writes
			\begin{equation}\label{matrix calL}
				\mathcal{L}_{c,\sigma,\gamma}=\begin{pmatrix}
					\left(c+\tfrac{\gamma}{2}\right)\partial_{x} & \tfrac{1}{2}|D|\vspace{0.2cm}\\
					\sigma-\gamma^2+\tfrac{\gamma^2}{2}|D|-\sigma|D|^2& \left(c+\tfrac{\gamma}{2}\right)\partial_{x}
				\end{pmatrix}.
			\end{equation}
			In particular, it is a Fourier multiplier. Its action on $(\hat{\eta},\hat{\psi}),$ admitting the Fourier expansions
			$$\hat{\eta}(x)=\sum_{n=1}^{\infty}a_n\cos(n\mathbf{m}x),\qquad\hat{\psi}(x)=\sum_{n=1}^{\infty}b_n\sin(n\mathbf{m}x),\qquad a_n,b_n\in\mathbb{R},$$
			is given by
			$$\mathcal{L}_{c,\sigma,\gamma}\begin{pmatrix}
				\hat{\eta}\\
				\hat{\psi}
			\end{pmatrix}(x)=\sum_{n=1}^{\infty}\begin{pmatrix}
				\sin(n\mathbf{m}x) & 0\\
				0 & \cos(n\mathbf{m}x)
			\end{pmatrix}M_{n\mathbf{m}}(c,\sigma,\gamma)\begin{pmatrix}
				a_n\\
				b_n
			\end{pmatrix},$$
			with
			\begin{equation}\label{def Mn}
				M_n(c,\sigma,\gamma)\triangleq
				\begin{pmatrix}
					-\left(c+\frac{\gamma}{2}\right)n & \frac{n}{2}\vspace{0.2cm}\\
					\sigma-\gamma^2+\frac{\gamma^2}{2}n-\sigma n^2 & \left(c+\frac{\gamma}{2}\right)n
				\end{pmatrix}.
			\end{equation}
			\item The operator
			$\mathcal{L}_{c,\sigma,\gamma}:X_{\mathbf{m}}^{s}\longrightarrow Y_{\mathbf{m}}^{s}$ is Fredholm with zero index.
		\end{enumerate}
	\end{prop}
	\begin{proof}
		$(i)$ First observe that from the expression \eqref{def K}, one readily gets
		\begin{equation}\label{diff K0}
			d_{\eta}\mathscr{K}(0)[\hat{\eta}]=\hat{\eta}+\hat{\eta}_{xx}=(\textnormal{Id}-|D|^2)\eta.
		\end{equation}
		Then, differentiating in \eqref{def H0}, we infer
		\begin{equation}\label{diff H0}
			d_{\eta}\mathscr{H}_0(0)[\hat{\eta}][f](x)=0.
		\end{equation}
		Now, differentiating in \eqref{def D0} and using one more time \eqref{trigo} yields
		\begin{equation}\label{diff D0}
			\begin{aligned}
				d_{\eta}\mathscr{D}_0(0)[\hat{\eta}][f](x)&=-\textnormal{p.v.}\int_{\mathbb{T}}\frac{\big(\hat{\eta}(y)-\hat{\eta}(x)\big)\cos(x-y)}{1-\cos(x-y)}f(y)\textnormal{d} y-\int_{\mathbb{T}}\big(\hat{\eta}(x)+\hat{\eta}(y)\big)f(y)\textnormal{d} y\\
				&=\frac{1}{2}\textnormal{p.v.}\int_{\mathbb{T}}\frac{\hat{\eta}(x)-\hat{\eta}(y)}{\sin^2\left(\frac{x-y}{2}\right)}f(y)\textnormal{d} y-2\hat{\eta}(x)\int_{\mathbb{T}}f(y)\textnormal{d} y.
			\end{aligned}
		\end{equation}
		For later purposes, we shall study the case $f=1$. Using Fourier symbols representation, we see that
		\begin{equation}\label{modD dxH}
			|D|=\partial_{x}\mathcal{H}=\mathcal{H}\partial_x.
		\end{equation}
		In addition, since $\cot'(x)=-\frac{1}{\sin^2(x)}$ and $\mathcal{H}[1]=0,$ we get
		\begin{equation}\label{calc dxH}
			\begin{aligned}
				\partial_{x}\mathcal{H}[f](x)&=\partial_{x}\mathcal{H}[f](x)-f(x)\partial_{x}\mathcal{H}[1](x)\\
				&=\textnormal{p.v.}\int_{\mathbb{T}}\partial_{x}\left[\cot\left(\frac{x-y}{2}\right)\right]\big(f(y)-f(x)\big)\textnormal{d} y\\
				&=\frac{1}{2}\textnormal{p.v.}\int_{\mathbb{T}}\frac{f(x)-f(y)}{\sin^2\left(\frac{x-y}{2}\right)}\textnormal{d} y.
			\end{aligned}
		\end{equation}
		Putting together \eqref{diff D0}, \eqref{modD dxH} and \eqref{calc dxH} yields
		\begin{equation}\label{diff D01}
			d_{\eta}\mathscr{D}_0(0)[\hat{\eta}][1]=|D|\hat{\eta}-2\hat{\eta}.
		\end{equation}
		Combining \eqref{def H} and \eqref{D00 K0} and \eqref{diff H0}, we get
		\begin{equation}\label{diff H}
			d_{\eta}\mathscr{H}(0)[\hat{\eta}][f]=\hat{\eta}_x\mathscr{D}_0(0)[f]+d_{\eta}\mathscr{H}_0(0)[f]=\hat{\eta}_{x}\int_{\mathbb{T}}f(y)\textnormal{d} y.
		\end{equation}
		As a consequence, differentiating \eqref{def F1} and using \eqref{H0}, \eqref{modD dxH} and \eqref{diff H} implies
		\begin{equation}\label{diff F1}
			\begin{aligned}
				d_{\eta,\psi}\mathcal{F}_1(c,\sigma,\gamma,0,0)[\hat{\eta},\hat{\psi}]&=c\,\hat{\eta}_x+\frac{1}{2}\mathscr{H}(0)[\hat{\psi}_x]+\frac{\gamma}{2}d_{\eta}\mathscr{H}(0)[\hat{\eta}][1]\\
				&=c\,\hat{\eta}_x+\frac{1}{2} \mathcal{H}\partial_{x}\hat{\psi}+\frac{\gamma}{2}\hat{\eta}_x\\
				&=\left(c+\frac{\gamma}{2}\right)\partial_{x}\hat{\eta}+\frac{|D|}{2}\hat{\psi}.
			\end{aligned}
		\end{equation}
		Besides, differentiating \eqref{def F2} and making appeal to \eqref{D00 K0}, \eqref{diff K0} and \eqref{diff D01}, we find
		\begin{equation}\label{diff F2}
			\begin{aligned}
				d_{\eta,\psi}\mathcal{F}_2(c,\sigma,\gamma,0,0)[\hat{\eta},\hat{\psi}]&=c\,\hat{\psi}_x+\frac{\gamma}{2}\hat{\psi}_x\mathscr{D}_0(0)[1]+\frac{\gamma^2}{2}d_{\eta}\mathscr{D}_0(0)[\hat{\eta}][1]+\frac{\gamma}{2}\mathscr{D}_0(0)[\hat{\psi}_x]+\sigma d_{\eta}\mathscr{K}(0)[\hat{\eta}]\\
				&=\left(c+\frac{\gamma}{2}\right)\partial_x\hat{\psi}+(\sigma-\gamma^2)\hat{\eta}+\frac{\gamma^2}{2}|D|\hat{\eta}-\sigma|D|^2\hat{\eta}.
			\end{aligned}
		\end{equation}
		Putting together \eqref{diff F1} and \eqref{diff F2}, we obtain the matrix representation \eqref{matrix calL} for the linearized operator.\\
		
		\noindent $(ii)$ Coming back to the expression \eqref{matrix calL}, we decompose the operator as follows
		$$\mathcal{L}_{c,\sigma,\gamma}=I_{\sigma}+K_{c,\sigma,\gamma},$$
		where
		\begin{align*}
			I_{\sigma}\triangleq
			\begin{pmatrix}
				0 & \frac{1}{2}|D|\\
				-\sigma|D|^2 & 0
			\end{pmatrix},\qquad K_{c,\sigma,\gamma}\triangleq
			\begin{pmatrix}
				\left(c+\tfrac{\gamma}{2}\right)\partial_{x} & 0\vspace{0.2cm}\\
				\sigma-\gamma^2+\tfrac{\gamma^2}{2}|D|& \left(c+\tfrac{\gamma}{2}\right)\partial_{x}
			\end{pmatrix}.
		\end{align*}
		Clearly, $I_{\sigma}:X_{\mathbf{m}}^{s}\to Y_{\mathbf{m}}^{s}$ is an isomorphism. In addition, $K_{c,\sigma,\gamma}:X_{\mathbf{m}}^{s}\to Y_{\mathbf{m}}^{s+\frac{1}{2}}$ is continuous. By Rellich-Kondrachov Theorem, we deduce that $K_{c,\sigma,\gamma}:X_{\mathbf{m}}^{s}\to Y_{\mathbf{m}}^{s}$ is a compact operator. This proves the claim by applying \cite[Cor. 5.9]{CR21}.
	\end{proof}
	\subsection{Bifurcation from $c$}\label{subsec c}
	In this subsection, we fix $\sigma>0$ and $\gamma\in\mathbb{R}$ (some restrictions will be imposed later on) and study the bifurcation from the parameter $c.$ Let us look for the values of the parameter $c$ such that the matrix $M_n(c,\sigma,\gamma),$ introduced in \eqref{def Mn}, is singular. For this aim, we compute its determinant
	\begin{equation}\label{determinant}
		\det\big(M_n(c,\sigma,\gamma)\big)=-\left(c+\frac{\gamma}{2}\right)^2n^2+\frac{n}{4}\left(2\sigma n^2-\gamma^2n+2(\gamma^2-\sigma)\right).
	\end{equation}
	We study the sign on $[1,\infty)$ of the polynomial function
	$$n\mapsto2\sigma n^2-\gamma^2n+2(\gamma^2-\sigma).$$
	The associated discriminant is
	$$\Delta=\gamma^4-16\sigma\gamma^2+16\sigma^2=(\gamma^2-8\sigma)^2-48\sigma^2.$$
	Hence, 
	$$2\sigma n^2-\gamma^2n+2\gamma^2>0\quad\Leftrightarrow\quad(n,\sigma,\gamma)\in\mathcal{S}\triangleq
	\mathcal{S}_1\cup\mathcal{S}_2,$$
	where
	\begin{align*}
		\mathcal{S}_1&\triangleq
		\big\{(n,\sigma,\gamma)\in\mathbb{N}^*\times(0,\infty)\times\mathbb{R}\quad\textnormal{s.t.}\quad4\sigma(2-\sqrt{3})<\gamma^2<4\sigma(2+\sqrt{3})\big\},\\
		\mathcal{S}_2&\triangleq
		\left\{(n,\sigma,\gamma)\in\mathbb{N}^*\times(0,\infty)\times\mathbb{R}\quad\begin{aligned}
			&\textnormal{s.t.}\quad\gamma^2\in[0,\infty)\setminus\left[4\sigma(2-\sqrt{3}),4\sigma(2+\sqrt{3})\right]\\
			&\textnormal{and}\quad 
			n\in\mathbb{R}\setminus[m_-(\sigma,\gamma),m_+(\sigma,\gamma)],\\
			&\textnormal{with}\quad m_{\pm}(\sigma,\gamma)\triangleq
			\frac{\gamma^2}{4\sigma}\pm\frac{1}{4\sigma}\sqrt{(\gamma^2-8\sigma)^2-48\sigma^2}
		\end{aligned}\right\}.
	\end{align*}
	Then, for $(n,\sigma,\gamma)\in\mathcal{S},$ we have
	\begin{equation}\label{def cnpm}
		\det\big(M_n(c,\sigma,\gamma)\big)=0\quad\Leftrightarrow\quad c=c_n^{\pm}(\sigma,\gamma)\triangleq
		-\frac{\gamma}{2}\pm\frac{1}{2}\sqrt{2\sigma n-\gamma^2+\frac{2(\gamma^2-\sigma)}{n}}\cdot
	\end{equation}
	\begin{prop}\label{prop hyp CR c}
		Let $(\mathbf{m},\sigma,\gamma)\in\mathcal{S}$ satisfying the additional condition
		\begin{equation}\label{cond c}
			\frac{\gamma^2-\sigma}{\sigma\mathbf{m}^2}\not\in\mathbb{N}^*.
		\end{equation}
		Then, the following properties hold true.
		\begin{enumerate}[label=(\roman*)]
			\item The kernel of the operator $\mathcal{L}_{c_{\mathbf{m}}^{\pm}(\sigma,\gamma),\sigma,\gamma}$ is one dimensional. More precisely
			$$\ker\left(\mathcal{L}_{c_{\mathbf{m}}^{\pm}(\sigma,\gamma),\sigma,\gamma}\right)=\mathtt{span}(x_{0,\sigma,\gamma,\mathbf{m}}^{\pm}),\qquad x_{0,\sigma,\gamma,\mathbf{m}}^{\pm}(x)\triangleq
			\begin{pmatrix}
				\cos(\mathbf{m}x)\\
				\pm\sqrt{2\sigma\mathbf{m}-\gamma^2+\frac{2(\gamma^2-\sigma)}{\mathbf{m}}}\sin(\mathbf{m}x)
			\end{pmatrix}.$$
			\item The operator $\mathcal{L}_{c_{\mathbf{m}}^{\pm}(\sigma,\gamma),\sigma,\gamma}:X_{\mathbf{m}}^{s}\longrightarrow Y_{\mathbf{m}}^{s}$ is Fredholm with zero index.
			\item The transversality condition holds, namely
			\begin{equation}\label{transversality}
				\partial_c\mathcal{L}_{c,\sigma,\gamma}|_{c=c_{\mathbf{m}}^{\pm}(\sigma,\gamma)}\left[x_{0,\sigma,\gamma,\mathbf{m}}^{\pm}\right]\not\in R\left(\mathcal{L}_{c_{\mathbf{m}}^{\pm}(\sigma,\gamma),\sigma,\gamma}\right).
			\end{equation}
		\end{enumerate}
	\end{prop}
	\begin{proof} 
		$(i)$ Let us study the spectral collisions of the $\mathbf{m}$-fold spectrum. More precisely, we shall solve
		\begin{equation}\label{eq spectral collision}
			c_{\mathbf{m}}^{\kappa_1}(\sigma,\gamma)=c_{k\mathbf{m}}^{\kappa_2}(\sigma,\gamma),\qquad (\kappa_1,\kappa_2)\in\{-,+\}^2,\qquad k\in\mathbb{N}^*.
		\end{equation}
		First observe that there can exist values of $k\in\mathbb{N}\setminus\{0,1\}$ such that $(k\mathbf{m},\sigma,\gamma)\not\in\mathcal{S}$. In this case, $c_{k\mathbf{m}}^{\kappa_2}(\sigma,\gamma)$ is either equal to $-\frac{\gamma}{2}$ or an element of $\mathbb{C}\setminus\mathbb{R}$. In both cases the equation \eqref{eq spectral collision} is not satisfied. So we can restrict the discussion to the case where $(k\mathbf{m},\sigma,\gamma)\in\mathcal{S}.$
		Coming back to the expression of $c_n^{\pm}(\sigma,\gamma)$ in \eqref{def cnpm}, the equation \eqref{eq spectral collision} implies
		$$\sigma\mathbf{m}+\frac{\gamma^2-\sigma}{\mathbf{m}}=\sigma k\mathbf{m}+\frac{\gamma^2-\sigma}{k\mathbf{m}}$$
		or equivalently
		$$\sigma\mathbf{m}^2(k-1)\left(k-\frac{\gamma^2-\sigma}{\sigma\mathbf{m}^2}\right)=\sigma\mathbf{m}^2k^2-\left(\sigma\mathbf{m}^2+\gamma^2-\sigma\right)k+\gamma^2-\sigma=0.$$
		Therefore, there are two solutions
		$$k_1=1,\qquad k_2=\frac{\gamma^2-\sigma}{\sigma\mathbf{m}^2}\cdot$$
		Since $c_{\mathbf{m}}^{-}(\sigma,\gamma)\neq c_{\mathbf{m}}^{+}(\sigma,\gamma),$ the case $k=k_1=1$ corresponds to the trivial solution. Thanks to the condition \eqref{cond c}, we deduce that the equation \eqref{eq spectral collision} has no non-trivial solution. As a consequence
		$$\det\Big(M_{\mathbf{m}}\big(c_{\mathbf{m}}^{\pm}(\sigma,\gamma),\sigma,\gamma\big)\Big)=0$$
		and
		$$\forall n\in\mathbb{N}\setminus\{0,1\},\quad\det\Big(M_{n\mathbf{m}}\big(c_{\mathbf{m}}^{\pm}(\sigma,\gamma),\sigma,\gamma\big)\Big)\neq0.$$
		This implies the one dimensional kernel property for the linearized operator and the generator is obtain by remarking that
		$$\begin{pmatrix}
			1\\
			\pm\sqrt{2\sigma\mathbf{m}-\gamma^2+\frac{2(\gamma^2-\sigma)}{\mathbf{m}}}
		\end{pmatrix}\in\ker\Big(M_{\mathbf{m}}\big(c_{\mathbf{m}}^{\pm}(\sigma,\gamma),\sigma,\gamma\big)\Big).$$
		
		\noindent $(ii)$ Follows from Proposition \ref{prop lin op}-$(ii).$\\
		
		\noindent $(iii)$ We introduce on $Y_{\mathbf{m}}^s$ the scalar product $\langle\cdot,\cdot\rangle$ defined as follows: for $(f,g)$ and $(\widetilde{f},\widetilde{g})$ in $Y_{\mathbf{m}}^s$ admitting the Fourier representations
		$$f(x)=\sum_{n=1}^{\infty}a_n\sin(n\mathbf{m}x),\qquad\widetilde{f}(x)=\sum_{n=1}^{\infty}\widetilde{a}_n\sin(n\mathbf{m}x)$$
		and
		$$g(x)=\sum_{n=1}^{\infty}b_n\cos(n\mathbf{m}x),\qquad\widetilde{g}(x)=\sum_{n=1}^{\infty}\widetilde{b}_n\cos(n\mathbf{m}x),$$
		with $a_n,\widetilde{a}_n,b_n,\widetilde{b}_n\in\mathbb{R}$, their scalar product is given by
		$$\langle(f,g),(\widetilde{f},\widetilde{g})\rangle\triangleq\sum_{n=1}^{\infty}a_n\widetilde{a}_n+b_n\widetilde{b}_n.$$
		Define
		$$y_{0,\sigma,\gamma,\mathbf{m}}^{\pm}(x)\triangleq
		\begin{pmatrix}
			\mp\sqrt{2\sigma\mathbf{m}-\gamma^2+\frac{2(\gamma^2-\sigma)}{\mathbf{m}}}\sin(\mathbf{m}x)\\
			\cos(\mathbf{m}x)
		\end{pmatrix}.$$
		Let us prove that 
		\begin{equation}\label{range}
			R\left(\mathcal{L}_{c_{\mathbf{m}}^{\pm}(\sigma,\gamma),\sigma,\gamma}\right)=\mathtt{span}\left(y_{0,\sigma,\gamma,\mathbf{m}}^{\pm}\right)^{\perp_{\langle\cdot,\cdot\rangle}}.
		\end{equation}
		Take an element $y\in R\left(\mathcal{L}_{c_{\mathbf{m}}^{\pm}(\sigma,\gamma),\sigma,\gamma}\right).$ By construction, 
		$$y(x)=\sum_{n=1}^{\infty}\begin{pmatrix}
			\sin(n\mathbf{m}x) & 0\\
			0 & \cos(n\mathbf{m}x)
		\end{pmatrix}M_{n\mathbf{m}}(c,\sigma,\gamma)\begin{pmatrix}
			a_n\\
			b_n
		\end{pmatrix}.$$
		Then,
		\begin{align*}
			\left\langle y,y_0^{\pm}\right\rangle&=\left\langle M_{\mathbf{m}}\left(c_{\mathbf{m}}^{\pm}(\sigma,\gamma),\sigma,\gamma\right)\begin{pmatrix}
				a_n\\
				b_n
			\end{pmatrix},\begin{pmatrix}
				\mp\sqrt{2\sigma\mathbf{m}-\gamma^2+\frac{2(\gamma^2-\sigma)}{\mathbf{m}}}\\
				1
			\end{pmatrix}\right\rangle_{\mathbb{R}^2}\\
			&=\left\langle \begin{pmatrix}
				a_n\\
				b_n
			\end{pmatrix},M_{\mathbf{m}}^{\top}\left(c_{\mathbf{m}}^{\pm}(\sigma,\gamma),\sigma,\gamma\right)\begin{pmatrix}
				\mp\sqrt{2\sigma\mathbf{m}-\gamma^2+\frac{2(\gamma^2-\sigma)}{\mathbf{m}}}\\
				1
			\end{pmatrix}\right\rangle_{\mathbb{R}^2}\\
			&=0.
		\end{align*}
		The last identity is obtained because by construction
		$$\begin{pmatrix}
			\mp\sqrt{2\sigma\mathbf{m}-\gamma^2+\frac{2(\gamma^2-\sigma)}{\mathbf{m}}}\\
			1
		\end{pmatrix}\in\ker\Big(M_{\mathbf{m}}^{\top}\big(c_{\mathbf{m}}^{\pm}(\sigma,\gamma),\sigma,\gamma\big)\Big).$$
		Recall that the notation $M^{\top}$ denotes the transposed of the matrix $M.$ This proves that
		\begin{equation}\label{first inclusion claim}
			R\left(\mathcal{L}_{c_{\mathbf{m}}^{\pm}(\sigma,\gamma),\sigma,\gamma}\right)\subset\mathtt{span}\left(y_{0,\sigma,\gamma,\mathbf{m}}^{\pm}\right)^{\perp_{\langle\cdot,\cdot\rangle}}.
		\end{equation}
		Now, since the space $\mathtt{span}\left(y_{0,\sigma,\gamma,\mathbf{m}}^{\pm}\right)$ is of finite dimension, then we can apply the orthogonal supplementary Theorem in the pre-Hilbertian space $(Y_{\mathbf{m}}^{s},\langle\cdot,\cdot\rangle)$ to get
		$$Y_{\mathbf{m}}^{s}=\mathtt{span}\left(y_{0,\sigma,\gamma,\mathbf{m}}^{\pm}\right)\overset{\perp}{\oplus}\mathtt{span}\left(y_{0,\sigma,\gamma,\mathbf{m}}^{\pm}\right)^{\perp_{\langle\cdot,\cdot\rangle}}.$$
		This proves that $\mathtt{span}\left(y_{0,\sigma,\gamma,\mathbf{m}}^{\pm}\right)^{\perp_{\langle\cdot,\cdot\rangle}}$ is of codimension one in $Y_{\mathbf{m}}^{s}.$ Besides, the points $(i)$ and $(ii)$ give that $R\left(\mathcal{L}_{c_{\mathbf{m}}^{\pm}(\sigma,\gamma),\sigma,\gamma}\right)$ is also of codimension one in $Y_{\mathbf{m}}^{s}.$ Together with the inclusion \eqref{first inclusion claim}, we conclude \eqref{range}. With this in hand, we can now check the transversality condition. Notice that, from \eqref{matrix calL}, we get
		$$\partial_{c}\mathcal{L}_{c,\sigma,\gamma}|_{c=c_{\mathbf{m}}^{\pm}(\sigma,\gamma)}=\begin{pmatrix}
			\partial_{x} & 0\\
			0 & \partial_{x}
		\end{pmatrix}.$$
		Thus, a straitforward computation together with the fact that $(\mathbf{m},\sigma,\gamma)\in\mathcal{S}$ give
		\begin{align*}
			\left\langle\partial_c\mathcal{L}_{c,\sigma,\gamma}|_{c=c_{\mathbf{m}}^{\pm}(\sigma,\gamma)}\left[x_{0,\sigma,\gamma,\mathbf{m}}^{\pm}\right],y_{0,\sigma,\gamma,\mathbf{m}}^{\pm}\right\rangle=\pm2\mathbf{m}\sqrt{2\sigma\mathbf{m}-\gamma^2+\frac{2(\gamma^2-\sigma)}{\mathbf{m}}}\neq0.
		\end{align*}
		According to \eqref{range}, this proves the transversality condition \eqref{transversality} and achieves the proof of Proposition \ref{prop hyp CR c}.
	\end{proof}
	\begin{proof}[Proof of Theorem \ref{thm bif vortex sheet}-$(i)$]
		We apply Theorem \ref{thm CR} together with Lemma \ref{lemma trivial sol} and Propositions \ref{prop regularity} and \ref{prop hyp CR c}.
	\end{proof}
	\subsection{Bifurcation from $\sigma$}\label{subsec sigma}
	In this subsection, we fix $(c,\gamma)\in\mathbb{R}^2$ (some restrictions will be imposed later on) and study the bifurcation from the parameter $\sigma.$ According to \eqref{determinant}, we have
	\begin{equation}\label{det M1}
		\det\big(M_1(c,\sigma,\gamma)\big)=-c(c+\gamma)
	\end{equation}
	is independent of $\sigma$ and, for any $n\geqslant2,$
	\begin{equation}\label{def sigman}
		\det\big(M_n(c,\sigma,\gamma)\big)=0\quad\Leftrightarrow\quad \sigma=\sigma_n(c,\gamma)\triangleq
		\frac{\left[\left(2c+\gamma\right)^2+\gamma^2\right]n-2\gamma^2}{2(n^2-1)}\cdot
	\end{equation}
	In the sequel, we denote
	$$\alpha(c,\gamma)\triangleq
	\left(2c+\gamma\right)^2+\gamma^2,\qquad\beta(\gamma)\triangleq2\gamma^2.$$
	The condition $\sigma>0$ requires
	$$n\alpha(c,\gamma)-\beta(\gamma)>0,\qquad\textnormal{i.e.}\qquad n>N(c,\gamma)\triangleq
	\frac{2\gamma^2}{\left(2c+\gamma\right)^2+\gamma^2}\cdot$$
	\begin{prop}\label{prop hyp CR sigma}
		Let $(c,\gamma)\in\mathbb{R}^2$ and $\mathbf{m}\in\mathbb{N}\setminus\{0,1\}$ with $\mathbf{m}>N(c,\gamma).$ Assume in addition that
		\begin{equation}\label{cond sigma}
			\frac{(2\mathbf{m}-1)\gamma^2-(2c+\gamma)^2}{\mathbf{m}(2c+\gamma)^2+(\mathbf{m}-2)\gamma^2}\not\in\mathbb{N}^*.
		\end{equation}
		Then, the following properties hold true.
		\begin{enumerate}[label=(\roman*)]
			\item The kernel of the operator $\mathcal{L}_{c,\sigma_{\mathbf{m}}(c,\gamma),\gamma}$ is one dimensional. More precisely,
			$$\ker\left(\mathcal{L}_{c,\sigma_{\mathbf{m}}(c,\gamma),\gamma}\right)=\mathtt{span}(x_{0,c,\gamma,\mathbf{m}}),\qquad x_{0,c,\gamma,\mathbf{m}}(x)\triangleq
			\begin{pmatrix}
				\cos(\mathbf{m}x)\\
				(2c+\gamma)\sin(\mathbf{m}x)
			\end{pmatrix}.$$
			\item The operator $\mathcal{L}_{c,\sigma_{\mathbf{m}}(c,\gamma),\gamma}:X_{\mathbf{m}}^{s}\longrightarrow Y_{\mathbf{m}}^{s}$ is Fredholm with zero index.
			\item The transversality condition holds, namely
			\begin{equation}\label{transversality sigma}
				\partial_\sigma\mathcal{L}_{c,\sigma,\gamma}|_{\sigma=\sigma_{\mathbf{m}}(c,\gamma)}\left[x_{0,c,\gamma,\mathbf{m}}\right]\not\in R\left(\mathcal{L}_{c,\sigma_{\mathbf{m}}(c,\gamma),\gamma}\right).
			\end{equation}
		\end{enumerate}
	\end{prop}
	\begin{proof}
		$(i)$ Let us study the spectral collisions. We need to solve
		\begin{equation}\label{spectral collision sigma}
			\sigma_{\mathbf{m}}(c,\gamma)=\sigma_{k\mathbf{m}}(c,\gamma),\qquad k\in\mathbb{N}^*.
		\end{equation}
		In what follows, we simply denote $\alpha\triangleq\alpha(c,\gamma)$ and $\beta\triangleq\beta(\gamma).$ Coming back to the expression \eqref{def sigman}, the previous equation is equivalent to
		$$(\alpha\mathbf{m}-\beta)(k^2\mathbf{m}^2-1)=(\alpha k\mathbf{m}-\beta)(\mathbf{m}^2-1),$$
		or again
		$$(\alpha\mathbf{m}-\beta)\mathbf{m}^2(k-1)\left(k-\frac{\beta\mathbf{m}-\alpha}{\mathbf{m}(\alpha\mathbf{m}-\beta)}\right)=(\alpha\mathbf{m}-\beta)\mathbf{m}^2k^2-\alpha\mathbf{m}(\mathbf{m}^2-1)k+\mathbf{m}(\beta\mathbf{m}-\alpha)=0.$$
		Therefore, the equation \eqref{spectral collision sigma} admits two solutions
		$$k_1=1\textnormal{ (trivial solution)},\qquad k_2=\frac{\beta\mathbf{m}-\alpha}{\alpha\mathbf{m}-\beta}=\frac{(2\mathbf{m}-1)\gamma^2-(2c+\gamma)^2}{\mathbf{m}(2c+\gamma)^2+(\mathbf{m}-2)\gamma^2}\cdot$$
		The condition \eqref{cond sigma} implies that the equation \eqref{spectral collision sigma} has no non-trivial solution. Then, we can conclude as in the proof of Proposition \ref{prop hyp CR c}-$(i).$\\
		
		\noindent $(ii)$ Follows from Proposition \ref{prop lin op}-$(ii)$.\\
		
		\noindent $(iii)$ Proceeding as in the proof of Proposition \ref{prop hyp CR c}-$(iii)$, we get
		$$R\left(\mathcal{L}_{c,\sigma_{\mathbf{m}(c,\gamma),\gamma}}\right)=\mathtt{span}(y_{0,c,\gamma,\mathbf{m}})^{\perp_{\langle\cdot,\cdot\rangle}},\qquad y_{0,c,\gamma,\mathbf{m}}(x)\triangleq
		\begin{pmatrix}
			-(2c+\gamma)\sin(\mathbf{m}x)\\
			\cos(\mathbf{m}x)
		\end{pmatrix}.$$
		From the expression \eqref{matrix calL}, we see that
		$$\partial_{\sigma}\mathcal{L}_{c,\sigma,\gamma}|_{\sigma=\sigma_{\mathbf{m}}(c,\gamma)}=\begin{pmatrix}
			0 & 0\\
			\textnormal{Id}-|D|^2 & 0
		\end{pmatrix}.$$
		Therefore, since $\mathbf{m}\neq1,$
		$$\left\langle\partial_{\sigma}\mathcal{L}_{c,\sigma,\gamma}|_{\sigma=\sigma_{\mathbf{m}}(c,\gamma)}[x_{0,c,\gamma,\mathbf{m}}],y_{0,c,\gamma,\mathbf{m}}\right\rangle=1-\mathbf{m}^2\neq0.$$
		This concludes \eqref{transversality sigma} and the proof of Proposition \ref{prop hyp CR sigma}.
	\end{proof}
	\begin{proof}[Proof of Theorem \ref{thm bif vortex sheet}-$(ii)$]
		We apply Theorem \ref{thm CR} together with Lemma \ref{lemma trivial sol} and Propositions \ref{prop regularity} and \ref{prop hyp CR sigma}.
	\end{proof}
	\subsection{Bifurcation from $\gamma$ for stationary vortex sheets}\label{subsec gamma}
	In this subsection, we fix $c=0$ and $\sigma>0.$ Then, we study the bifurcation from the parameter $\gamma.$ According to \eqref{det M1}, we get
	$$\det\big(M_1(0,\sigma,\gamma)\big)=0,$$
	and in view of \eqref{determinant}, for any $n\geqslant2$, we have
	$$\det\big(M_n(0,\sigma,\gamma)\big)=0\quad\Leftrightarrow\quad \gamma=\gamma_n^{\pm}(\sigma)\triangleq
	\pm\sqrt{\sigma(n+1)}.$$
	\begin{prop}\label{prop hyp CR gamma}
		Let $\sigma>0$ and $\mathbf{m}\in\mathbb{N}\setminus\{0,1\}.$ Then, the following properties hold true.
		\begin{enumerate}[label=(\roman*)]
			\item The kernel of the operator $\mathcal{L}_{0,\sigma,\gamma_{\mathbf{m}}^{\pm}(\sigma)}$ is one dimensional. More precisely,
			$$\ker\left(\mathcal{L}_{0,\sigma,\gamma_{\mathbf{m}}^{\pm}(\sigma)}\right)=\mathtt{span}(x_{0,\sigma,\mathbf{m}}^{\pm}),\qquad x_{0,\sigma,\mathbf{m}}^{\pm}(x)\triangleq
			\begin{pmatrix}
				\cos(\mathbf{m}x)\\
				\gamma_{\mathbf{m}}^{\pm}(\sigma)\sin(\mathbf{m}x)
			\end{pmatrix}.$$
			\item The operator $\mathcal{L}_{0,\sigma,\gamma_{\mathbf{m}}^{\pm}(\sigma)}:X_{\mathbf{m}}^{s}\longrightarrow Y_{\mathbf{m}}^{s}$ is Fredholm with zero index.
			\item The transversality condition holds, namely
			\begin{equation}\label{transversality gamma}
				\partial_{\gamma}\mathcal{L}_{0,\sigma,\gamma}|_{\gamma=\gamma_{\mathbf{m}}^{\pm}(\sigma)}\left[x_{0,\sigma,\mathbf{m}}^{\pm}\right]\not\in R\left(\mathcal{L}_{0,\sigma,\gamma^{\pm}_{\mathbf{m}}(\sigma)}\right).
			\end{equation}
		\end{enumerate}
	\end{prop}
	\begin{proof}
		$(i)$ Observe that $\gamma_{\mathbf{m}}^-(\sigma)=-\gamma_{\mathbf{m}}^+(\sigma)\neq0$ and that the sequence $\big(\gamma_{n}^{+}(\sigma)\big)_{n\geqslant2}$ is strictly increasing. This immediatly prevents spectral collisions and allows to conclude similarly to Proposition \ref{prop hyp CR c}-$(i)$.\\
		
		\noindent $(ii)$ Follows from Proposition \ref{prop lin op}-(ii).\\
		
		\noindent $(iii)$ Proceeding as in the proof of Proposition \ref{prop hyp CR c}-$(iii)$, we get
		$$R\left(\mathcal{L}_{c,\sigma,\gamma_{\mathbf{m}^{\pm}(\sigma)}}\right)=\mathtt{span}(y_{0,\sigma,\mathbf{m}}^{\pm})^{\perp_{\langle\cdot,\cdot\rangle}},\qquad y_{0,\sigma,\mathbf{m}}^{\pm}(x)\triangleq
		\begin{pmatrix}
			-\gamma_{\mathbf{m}}^{\pm}(\sigma)\sin(\mathbf{m}x)\\
			\cos(\mathbf{m}x)
		\end{pmatrix}.$$
		From \eqref{matrix calL}, we have
		$$\partial_{\gamma}\mathcal{L}_{0,\sigma,\gamma}|_{\gamma=\gamma_{\mathbf{m}}^{\pm}(\sigma)}=\begin{pmatrix}
			\frac{1}{2}\partial_{x} & 0\vspace{0.2cm}\\
			\gamma_{\mathbf{m}}^{\pm}(\sigma)(|D|-2) & \frac{1}{2}\partial_{x}
		\end{pmatrix}.$$
		Therefore, since $\mathbf{m}\neq1$ and $\gamma_{\mathbf{m}}^{\pm}(\sigma)\neq0,$
		$$\left\langle\partial_{\gamma}\mathcal{L}_{c,\sigma,\gamma}|_{\gamma=\gamma_{\mathbf{m}}^{\pm}(\sigma)}[x_{0,\sigma,\mathbf{m}}^{\pm}],y_{0,\sigma,\mathbf{m}}^{\pm}\right\rangle=2\gamma_{\mathbf{m}}^{\pm}(\sigma)\left(\mathbf{m}-1\right)\neq0.$$
		This implies \eqref{transversality gamma} and achieves the proof of Proposition \ref{prop hyp CR gamma}.
	\end{proof}
	\begin{proof}[Proof of Theorem \ref{thm bif vortex sheet}-$(iii)$]
		We apply Theorem \ref{thm CR} together with Lemma \ref{lemma trivial sol} and Propositions \ref{prop regularity} and \ref{prop hyp CR gamma}.
	\end{proof}

	\appendix
	\section{Appendix}
	
	\subsection{Derivation of \eqref{VS system} from \eqref{Euler} }
	\label{sec:contour_derivation}

	From \eqref{eq:upm_BS} it is immediate that we can recast \eqref{Euler} as an evolutionary equation on the interface only, since the bulk quantities $u^\pm$ and $p^\pm$ can be recovered from the interface evolution of $\omega$ and $z$. Such derivation is quite well-known in the literature (cf. \cite{FIL16,CCG10}) but we perform here detailed computations for the sake of clarity.\\

	Given two functions $f^{\pm}:\Omega^{\pm}(t)\to\mathbb{R}$ we define
	\begin{align*}
		\bbra{f^{\pm}}\triangleq f^--f^+.
	\end{align*}
	Let us now define the trace of the velocity $u^\pm$ (cf. \eqref{eq:upm_BS}) as
	\begin{align}
		\label{eq:trace_velocity}
		v^{\pm}(t,x)\triangleq\left.u^{\pm}\right|_{\Gamma(t)}(t,x)=u^{\pm}\big(t,z(t,x)\big),\qquad(t,x)\in(0,T)\times\mathbb{T}. 
	\end{align}
	From \eqref{eq:upm_BS} we can compute $ v^\pm $ defined in \eqref{eq:trace_velocity} as 
	\begin{equation*}
		v^{\pm}\pare{t, x} = \lim _{\epsilon\searrow 0} u^{\pm}\pare{t, z\pare{t, x} \pm \epsilon z_x^\perp\pare{t, x}}
	\end{equation*}
	since $ z\pare{t, x} \pm \epsilon z_x^\perp\pare{t, x}\in \Omega^{\pm}\pare{t} $ for any $ x\in \mathbb{T} $,
	so that the trace of the velocity flow $ v^{\pm} $ can be recasted in terms of $ z $ and $ \omega $ via the \emph{Birkhoff-Rott integral operator} 
	\begin{align}\label{eq:trace_BR}
		v^{\pm}=\textnormal{BR}(z)\omega \mp\frac{\omega}{2}\frac{z_x}{\av{z_x}^2},\qquad\textnormal{BR}(z) \omega\pare{t,x}\triangleq\frac{1}{2\pi}\int_\mathbb{R}\frac{\pare{z\pare{t,x}-z\pare{t,y}}^\perp}{\av{z\pare{t,x}-z\pare{t,y}}^2} \omega \pare{t, y}\textnormal{d}y. 
	\end{align}
	From \eqref{eq:trace_BR} we can deduce hence a relation between the vorticity strength and the trace of the velocity valid in $ \pare{0, T}\times \mathbb{T} $ (cf. \eqref{eq:trace_velocity})
	\begin{equation}
		\label{eq:omega_vpm}
		\omega = \bbra{ v^{\pm} }  \cdot z_x. 
	\end{equation}

	\begin{remark}
		Notice that from \eqref{eq:trace_BR} it is immediate that the normal (to $ \Gamma\pare{t} $) component of the velocity flow is continuous through the interface. 
	\end{remark}

	The relation \eqref{eq:omega_vpm} allows us to express $ \omega $ in terms of the trace of the velocity flow, hence, to derive the evolution equation for $ \omega $ by taking the tangential (to $ \Gamma\pare{t} $) trace of the first equation of  \eqref{Euler} onto $ \Gamma\pare{t} $. This procedure produces the evolutionary equations for $ v^\pm $ which are computed in \cite[Eq. (2.2)]{FIL16} and are given by
	\begin{align}\label{eq:trace_velocities_evolution}
		\pare{v^{\pm} \cdot z_x}_t - \pare{v^\pm \cdot z_t}_x + \frac{1}{2}\pare{\av{v^{\pm}}^2}_x + \pare{\left. p^\pm \right|_{\Gamma\pare{t}}}_x + \pare{z_2}_x = 0,
	\end{align}
	where $z_2$ is the second component of the parametrization vector $z$. Using \eqref{eq:trace_velocities_evolution} we can hence compute the evolution equation for $\omega$ defined as in \eqref{eq:omega_vpm} which is, using the continuity of the stress tensor among the surface $\Gamma(t)$, i.e. the third equation of \eqref{Euler}
	\begin{equation}\label{eq:omega_consaverage1}
		\omega_t-\pare{\bbra{v^\pm}\cdot z_t}_x+\frac{1}{2}\pare{\bbra{\av{v^\pm}^2}}_x+\sigma\big(\mathcal{K}(z)\big)_x=0, 
	\end{equation}
	while we use \eqref{eq:trace_BR} in order to obtain the identity
	\begin{equation*}
		v^+=v^- - \omega \ \frac{z_x}{\av{z_x}^2}, 
	\end{equation*}
	from which we derive
	\begin{align}
		\label{eq:various_trace_relations}
		v^+\cdot z_x = v^-\cdot z_x - \omega , 
		&&
		v^+\cdot z_t = v^-\cdot z_t - \omega\frac{z_t\cdot z_x}{\av{z_x}^2} , 
		&&
		\av{v^+}^2 = \av{v^-}^2 + \frac{\omega^2}{\av{z_x}^2} - 2\omega \ \frac{v^-\cdot z_x}{\av{z_x}^2}\cdot 
	\end{align}
	The relations in \eqref{eq:various_trace_relations} and \eqref{eq:trace_BR} give us that
	\begin{equation}\label{eq:omega_consaverage2}
		-\pare{\bbra{v^\pm}\cdot z_t}_x+\frac{1}{2}\pare{\bbra{\av{v^\pm}^2}}_x=-\frac{1}{2}\pare{\frac{\omega^2}{\av{z_x}^2}}_x-\pare{\omega \ \frac{\pare{z_t-v^-}\cdot z_x}{\av{z_x}^2}}_x\\
		=-\pare{\omega\ \frac{\pare{z_t-\textnormal{BR}(z)\omega}\cdot z_x}{\av{z_x}^2}}_x.
	\end{equation}
	The system \eqref{Euler} can thus be recasted as an evolutionary equation on the interface only via the unknowns $ \pare{\omega, z} $, plugging \eqref{eq:omega_consaverage2} in \eqref{eq:omega_consaverage1}, thus obtaining \eqref{VS system}.
	
	\subsection{Derivation of \eqref{eq:KH2} from \eqref{VS system} }\label{appendix derive eq}
	Here we explain how to get the system \eqref{eq:KH2} from \eqref{VS system}.
	Let us consider a parametrization of $\Gamma(t)$ in the form
	$$z(t,x)=R(t,x)e^{\ii x},\qquad R(t,x)\triangleq
	\sqrt{1+2\eta(t,x)}.$$
	Straightforward calculations show that
	\begin{equation}\label{basic idtt}
		\begin{aligned}
			z_{t}(t,x)&=\frac{\eta_t(t,x)}{R(t,x)}e^{\ii x},\\
			z_{x}(t,x)&=\left(\frac{\eta_x(t,x)}{R(t,x)}+\ii R(t,x)\right)e^{\ii x},\\
			z_{x}^{\perp}(t,x)&=\left(\ii\frac{\eta_x(t,x)}{R(t,x)}-R(t,x)\right)e^{\ii x},\\
			z_{xx}(t,x)&=\left(\frac{\eta_{xx}(t,x)}{R(t,x)}-\frac{\eta_{x}^2(t,x)}{R^3(t,x)}+2\ii\frac{\eta_{x}(t,x)}{R(t,x)}-R(t,x)\right)e^{\ii x}.
		\end{aligned}
	\end{equation}
	As a consequence,
	\begin{equation}\label{basic idtt2}
		\begin{aligned}
			z_t\cdot z_{x}^{\perp}&=-\eta_{t},\\
			z_{t}\cdot z_{x}&=\frac{\eta_{t}\eta_x}{R^2},\\
			z_{x}^{\perp}\cdot z_{xx}&=3\left(\frac{\eta_x}{R}\right)^2-\eta_{xx}+R^2,\\
			|z_x|^2&=R^2+\left(\frac{\eta_{x}}{R}\right)^2.
		\end{aligned}
	\end{equation}
	The last two identities combined with \eqref{def K} give immediately
	\begin{equation}\label{curv}
		\mathcal{K}(z)=\frac{\eta_{xx}-2\left(\frac{\eta_{x}}{R}\right)^2}{\left(R^2+\left(\frac{\eta_{x}}{R}\right)^2\right)^{\frac{3}{2}}}-\left(R^2+\left(\frac{\eta_{x}}{R}\right)^2\right)^{-\frac{1}{2}}=\mathscr{K}(\eta).
	\end{equation}
	In addition,
	\begin{equation}\label{modul}
		\begin{aligned}
			\big(z(x)-z(y)\big)\cdot z_{x}(x)&=\eta_{x}(x)\left(1-\frac{R(y)}{R(x)}\cos(x-y)\right)+R(x)R(y)\sin(x-y),\\
			|z(x)-z(y)|^2&=2\big(1+\eta(x)+\eta(y)-R(x)R(y)\cos(x-y)\big).
		\end{aligned}
	\end{equation}
	Using \eqref{modul} and the notation \eqref{def H}, we deduce that
	\begin{align*}
		\textnormal{BR}(z)\omega\cdot z_{x}^{\perp}(x)&=\textnormal{p.v.}\int_{\mathbb{T}}\frac{\big(z(x)-z(y)\big)^{\perp}\cdot z_{x}^{\perp}(x)}{|z(x)-z(y)|^2}\omega(y)\textnormal{d} y\\
		&=\textnormal{p.v.}\int_{\mathbb{T}}\frac{\big(z(x)-z(y)\big)\cdot z_{x}(x)}{|z(x)-z(y)|^2}\omega(y)\textnormal{d} y\\
		&=\tfrac{1}{2}\mathscr{H}(\eta)[\omega](x).
	\end{align*}
	Together with \eqref{basic idtt2}, the first equation in \eqref{VS system} becomes
	\begin{equation}\label{eq1}
		\eta_{t}=-\tfrac{1}{2}\mathscr{H}(\eta)[\omega].
	\end{equation}
	Combining \eqref{VS system}, \eqref{basic idtt} and \eqref{basic idtt2}, we find
	\begin{align*}
		z_{t}\cdot z_{x}=-\frac{\eta_{x}}{R^2}\textnormal{BR}(z)\omega\cdot z_{x}^{\perp}.
	\end{align*}
	Thus,
	\begin{align*}
		\big(z_{t}-\textnormal{BR}(z)\omega\big)\cdot z_{x}&=-\textnormal{BR}(z)\omega\cdot\left(\frac{\eta_{x}}{R^2}z_{x}^{\perp}+z_{x}\right)\\
		&=-|z_{x}|^2\textnormal{BR}(z)\omega\cdot\left(\frac{\ii e^{\ii x}}{R}\right).
	\end{align*}
	But
	$$\big(z(x)-z(y)\big)^{\perp}\cdot\left(\frac{\ii e^{\ii x}}{R}\right)=1-\frac{R(y)}{R(x)}\cos(x-y).$$
	The last computations with \eqref{modul} and \eqref{def D0} give
	\begin{align*}
		\omega\frac{\big(z_{t}-\textnormal{BR}(z)\omega\big)\cdot z_{x}}{|z_x|^2}(x)&=-\omega(x)\textnormal{p.v.}\int_{\mathbb{T}}\frac{\big(z(x)-z(y)\big)^{\perp}\cdot\left(\tfrac{\ii e^{\ii x}}{R}\right)}{|z(x)-z(y)|^2}\omega(y)\textnormal{d} y\\
		&=-\tfrac{1}{2}\omega(x)\mathscr{D}_{0}(\eta)[\omega](x).
	\end{align*}
	This together with \eqref{VS system} and \eqref{curv} imply
	$$\omega_{t}=\Big(-\tfrac{1}{2}\omega\mathscr{D}_{0}(\eta)[\omega]-\sigma\mathscr{K}(\eta)\Big)_x.$$
	
	\subsection{Crandall-Rabinowitz Theorem}
	We state here the local bifurcation result obtained in \cite{CR71} and used in this study to construct our solutions.
	\begin{theo}\label{thm CR}
		\textbf{(Crandall-Rabinowitz)}\\
		Let $X$ and $Y$ be two Banach spaces. Let $(p_0,u_0)\in\mathbb{R}\times X$ and $U$ be a neighborhood of $(p_0,u_0)$ in $\mathbb{R}\times X.$ Consider a $C^1$ function $F:U\rightarrow Y$ such that
		\begin{itemize}
			\item [$(1)$] $\forall(p,u_0)\in U,\quad F(p,u_0)=0.$
			\item [$(2)$] The operator $d_uF(p_0,u_0)$ is a Fredholm operator with zero index and such that
			$$\ker\big(d_uF(p_0,u_0)\big)=\mathtt{span}(x_0).$$
			\item [$(3)$] Transversality:
			$$\partial_{p}d_uF(p_0,u_0)[x_0]\not\in R\big(d_uF(p_0,u_0)\big).$$
		\end{itemize}
		If we decompose
		$$X=\mathtt{span}(x_0)\oplus Z,$$
		then there exist two $C^1$ functions
		$$p:(-\epsilon,\epsilon)\rightarrow\mathbb{R}\qquad\textnormal{and}\qquad z:(-\epsilon,\epsilon)\rightarrow Z,\qquad\textnormal{with}\qquad\epsilon>0,$$
		such that
		$$p(0)=p_0,\qquad z(0)=0$$
		and the set of zeros of $F$ in $U$ is the union of two curves 
		$$\big\{(p,u)\in U\quad\textnormal{s.t.}\quad F(p,u)=0\big\}=\big\{(p,u_0)\in U\big\}\cup\mathscr{C}_{\textnormal{\tiny{local}}},\qquad\mathscr{C}_{\textnormal{\tiny{local}}}\triangleq
		\big\{\big(p(\mathtt{s}),u_0+\mathtt{s}x_0+\mathtt{s}z(\mathtt{s})\big),\quad|\mathtt{s}|<\epsilon\big\}.$$
	\end{theo}
	
	\begin{footnotesize}
		\bibliography{references}

\begin{thebibliography}{10}

\bibitem{Ambrose03}
David~M. Ambrose.
\newblock Well-posedness of vortex sheets with surface tension.
\newblock {\em SIAM Journal on Mathematical Analysis}, 35(1):211--244, 2003.

\bibitem{Ambrose07}
David~M. Ambrose.
\newblock Regularization of the {K}elvin-{H}elmholtz instability by surface
  tension.
\newblock {\em Philosophical Transactions of the Royal Society of London.
  Series A. Mathematical, Physical and Engineering Sciences},
  365(1858):2253--2266, 2007.

\bibitem{AM07}
David~M. Ambrose and Nader Masmoudi.
\newblock Well-posedness of 3{D} vortex sheets with surface tension.
\newblock {\em Communications in Mathematical Sciences}, 5(2):391--430, 2007.

\bibitem{BBHM18}
Pietro Baldi, Massimiliano Berti, Emanuele Haus, and Riccardo Montalto.
\newblock Time quasi-periodic gravity water waves in finite depth.
\newblock {\em Inventiones Mathematicae}, 214(2):739–911, 2018.

\bibitem{BJLM24}
Pietro Baldi, Vesa Julin, and Domenico~Angelo La~Manna.
\newblock Liquid drop with capillarity and rotating traveling waves.
\newblock {\em arXiv preprint arXiv:2408.02333}.

\bibitem{BB97}
Thomas~Brooke Benjamin and Thomas~J. Bridges.
\newblock Reappraisal of the {K}elvin–{H}elmholtz problem. {P}art 1.
  {H}amiltonian structure.
\newblock {\em Journal of Fluid Mechanics}, 333:301--325, 1997.

\bibitem{BCGS2023}
Massimiliano Berti, Scipio Cuccagna, Francisco Gancedo, and Stefano Scrobogna.
\newblock Paralinearization and extended lifespan for solutions of the $ \alpha
  $-{SQG} sharp front equation.
\newblock {\em arXiv preprint arXiv:2310.15963}, 2023.

\bibitem{BD2016}
Massimiliano Berti and Jean-Marc Delort.
\newblock {\em Almost global solutions of capillary-gravity water waves
  equations on the circle}, volume~24 of {\em Lecture Notes of the Unione
  Matematica Italiana}.
\newblock Springer, Cham; Unione Matematica Italiana, [Bologna], 2018.

\bibitem{BFM21}
Massimiliano Berti, Luca Franzoi, and Alberto Maspero.
\newblock Traveling quasi-periodic water waves with constant vorticity.
\newblock {\em Archive for Rational Mechanics and Analysis}, 240:99–202,
  2021.

\bibitem{BFM24}
Massimiliano Berti, Luca Franzoi, and Alberto Maspero.
\newblock Pure gravity traveling quasi-periodic water waves with constant
  vorticity.
\newblock {\em Communications on Pure and Applied Mathematics},
  77(2):990--1064, 2024.

\bibitem{BHM22}
Massimiliano Berti, Zineb Hassainia, and Nader Masmoudi.
\newblock Time quasi-periodic vortex patches.
\newblock {\em Inventiones Mathematicae}, 233(3):1--113, 2023.

\bibitem{BMM2022}
Massimiliano Berti, Alberto Maspero, and Federico Murgante.
\newblock Hamiltonian {B}irkhoff normal form for gravity-capillary water waves
  with constant vorticity: almost global existence.
\newblock {\em arXiv preprint arXiv:2212.12255}, 2022.

\bibitem{BMM2023}
Massimiliano Berti, Alberto Maspero, and Federico Murgante.
\newblock Hamiltonian paradifferential {B}irkhoff normal form for water waves.
\newblock {\em Regular and Chaotic Dynamics. International Scientific Journal},
  28(4-5):543--560, 2023.

\bibitem{BM20}
Massimiliano Berti and Riccardo Montalto.
\newblock Quasi-periodic standing wave solutions of gravity-capillary water
  waves.
\newblock {\em Memoirs of the American Mathematical Society}, 263(1273), 2020.

\bibitem{Burbea82}
Jacob Burbea.
\newblock Motions of vortex patches.
\newblock {\em Letters in Mathematical Physics,}, 6(1):1--16, 1982.

\bibitem{CO1989}
Russel~E. Caflisch and Oscar~F. Orellana.
\newblock Singular solutions and ill-posedness for the evolution of vortex
  sheets.
\newblock {\em SIAM Journal on Mathematical Analysis}, 20(2):293--307, 1989.

\bibitem{CQZ2023a}
Daomin Cao, Guolin Qin, and Changjun Zou.
\newblock Co-rotating and traveling vortex sheets for the 2{D} incompressible
  {E}uler equation.
\newblock {\em Nonlinear Analysis}, 228:113186, 2023.

\bibitem{CQZ2023b}
Daomin Cao, Guolin Qin, and Changjun Zou.
\newblock Existence of stationary vortex sheets for the 2{D} incompressible
  {E}uler equation.
\newblock {\em Canadian Journal of Mathematics}, 75(3):828--853, 2023.

\bibitem{CCG2012}
Angel Castro, Diego C\'{o}rdoba, and Francisco Gancedo.
\newblock A naive parametrization for the vortex-sheet problem.
\newblock In {\em Mathematical aspects of fluid mechanics}, volume 402 of {\em
  London Mathematical Society Lecture Note Series}, pages 88--115. Cambridge
  University Press, Cambridge, 2012.

\bibitem{CCG2016}
Angel Castro, Diego C{\'o}rdoba, and Javier G{\'o}mez-Serrano.
\newblock Uniformly rotating analytic global patch solutions for active
  scalars.
\newblock {\em Annals of PDE}, 2(1):Art. 1, 34, 2016.

\bibitem{CCG2019}
Angel Castro, Diego C{\'o}rdoba, and Javier G{\'o}mez-Serrano.
\newblock Uniformly rotating smooth solutions for the incompressible 2{D}
  {E}uler equations.
\newblock {\em Archive for Rational Mechanics and Analysis}, 231(2):719--785,
  2019.

\bibitem{CCS08}
Ching-Hsiao~Arthur Cheng, Daniel Coutand, and Steve Shkoller.
\newblock On the motion of vortex sheets with surface tension in
  three-dimensional {E}uler equations with vorticity.
\newblock {\em Communications on Pure and Applied Mathematics},
  61(12):1715--1752, 2008.

\bibitem{CR21}
Christophe Cheverry and Nicolas Raymond.
\newblock {\em A guide to spectral theory Applications and exercises}.
\newblock Birkh\"{a}user/Springer, Cham, 2021.

\bibitem{CCG10}
Antonio C\'{o}rdoba, Diego C\'{o}rdoba, and Francisco Gancedo.
\newblock Interface evolution: water waves in 2-{D}.
\newblock {\em Advances in Mathematics}, 223(1):120--173, 2010.

\bibitem{CR71}
Michael~G. Crandall and Paul~H. Rabinowitz.
\newblock Bifurcation from simple eigenvalues.
\newblock {\em Journal of Functional Analysis}, 8:321--340, 1971.

\bibitem{DHHM16}
Francisco de~la Hoz, Zineb Hassainia, Taoufik Hmidi, and Joan Mateu.
\newblock An analytical and numerical study of steady patches in the disc.
\newblock {\em Analysis and PDE}, 9(7):1609--1670, 2016.

\bibitem{HHMV16}
Francisco de~la Hoz, Taoufik Hmidi, Joan Mateu, and Joan Verdera.
\newblock Doubly connected {$V$}-states for the planar {E}uler equations.
\newblock {\em SIAM Journal on Mathematical Analysis}, 48(3):1892--1928, 2016.

\bibitem{Delort1991}
Jean-Marc Delort.
\newblock Existence de nappes de tourbillon en dimension deux.
\newblock {\em Journal of the American Mathematical Society}, 4(3):553--586,
  1991.

\bibitem{FIL16}
Charles Fefferman, Alexandru~D. Ionescu, and Victor Lie.
\newblock On the absence of splash singularities in the case of two-fluid
  interfaces.
\newblock {\em Duke Mathematical Journal}, 165(3):417--462, 2016.

\bibitem{FG20a}
Roberto Feola and Filippo Giuliani.
\newblock Quasi-periodic traveling waves on an infinitely deep perfect fluid
  under gravity.
\newblock {\em arXiv preprint arXiv:2005.08280. To appear in Memoirs of the
  American Mathematical Society}, 2020.

\bibitem{FG20b}
Roberto Feola and Filippo Giuliani.
\newblock Time quasi-periodic traveling gravity water waves in infinite depth.
\newblock {\em Atti della Accademia nazionale dei Lincei. Rendiconti Lincei.
  Matematica e Applicazioni}, 31(4):901–916, 2020.

\bibitem{GHR2023}
Claudia Garcia, Zineb Hassainia, and Emeric Roulley.
\newblock Dynamics of vortex cap solutions on the rotating unit sphere.
\newblock {\em arXiv preprint arXiv:2306.00154}, 2023.

\bibitem{GHM2024}
Claudia García, Taoufik Hmidi, and Joan Mateu.
\newblock Time periodic solutions close to localized radial monotone profiles
  for the 2{D} {E}uler equations.
\newblock {\em Annals of PDE}, 10(1), 2024.

\bibitem{GIP23}
Javier G{\'o}mez-Serrano, Alexandru~D. Ionescu, and Jaemin Park.
\newblock Quasiperiodic solutions of the generalized {SQG} equation.
\newblock {\em arXiv preprint arXiv:2303.03992}, 2023.

\bibitem{GSPSY21}
Javier G\'{o}mez-Serrano, Jaemin Park, Jia Shi, and Yao Yao.
\newblock Remarks on stationary and uniformly-rotating vortex sheets: rigidity
  results.
\newblock {\em Communications in Mathematical Physics}, 386(3):1845--1879,
  2021.

\bibitem{GSPSY22}
Javier G\'{o}mez-Serrano, Jaemin Park, Jia Shi, and Yao Yao.
\newblock Remarks on stationary and uniformly rotating vortex sheets:
  flexibility results.
\newblock {\em Philosophical Transactions of the Royal Society A. Mathematical,
  Physical and Engineering Sciences}, 380(2226):Paper No. 20210045, 15, 2022.

\bibitem{HHR23}
Zineb Hassainia, Taoufik Hmidi, and Emeric Roulley.
\newblock Invariant {KAM} tori around annular vortex patches for 2{D} {E}uler
  equations.
\newblock {\em arXiv preprint arXiv:2302.01311}, 2023.

\bibitem{HMW20}
Zineb Hassainia, Nader Masmoudi, and Miles~H. Wheeler.
\newblock Global bifurcation of rotating vortex patches.
\newblock {\em Communications on Pure and Applied Mathematics},
  73(9):1933--1980, 2020.

\bibitem{HR22}
Zineb Hassainia and Emeric Roulley.
\newblock Boundary effects on the emergence of quasi-periodic solutions for
  {E}uler equations.
\newblock {\em arXiv preprint arXiv:2202.10053}, 2022.

\bibitem{Helmholtz1858}
Hermann~von Helmholtz.
\newblock Über integrale der hydrodynamischen gleichungen, welche den
  wirbelbewegungen entsprechen.
\newblock {\em Journal für die reine und angewandte Mathematik}, 55:25--55,
  1858.

\bibitem{HHRZ24}
Taoufik Hmidi, Haroune Houamed, Emeric Roulley, and Mohamed Zerguine.
\newblock Uniformly rotating vortices for the lake equation.
\newblock {\em arXiv preprint arXiv:2401.14273}, 2024.

\bibitem{HM16b}
Taoufik Hmidi and Joan Mateu.
\newblock Bifurcation of rotating patches from {K}irchhoff vortices.
\newblock {\em Discrete and Continuous Dynamical Systems}, 36(10):5401--5422,
  2016.

\bibitem{HM16a}
Taoufik Hmidi and Joan Mateu.
\newblock Degenerate bifurcation of the rotating patches.
\newblock {\em Advances in Mathematics}, 302:799--850, 2016.

\bibitem{HM13}
Taoufik Hmidi, Joan Mateu, and Joan Verdera.
\newblock Boundary regularity of rotating vortex patches.
\newblock {\em Archive for Rational Mechanics and Analysis}, 209(1):171--208,
  2013.

\bibitem{HMV15}
Taoufik Hmidi, Joan Mateu, and Joan Verdera.
\newblock On rotating doubly connected vortices.
\newblock {\em Journal of Differential Equations}, 258(4):1395--1429, 2015.

\bibitem{HR21}
Taoufik Hmidi and Emeric Roulley.
\newblock Time quasi-periodic vortex patches for quasi-geostrophic
  shallow-water equations.
\newblock {\em arXiv preprint arXiv:2110.13751. To appear in Mémoires de la
  Société Mathématique de France}, 2021.

\bibitem{HXX2023}
Taoufik Hmidi, Liutang Xue, and Zhilong Xue.
\newblock Unified theory on {V}-states structures for active scalar equations.
\newblock {\em arXiv preprint arXiv:2312.02874}, 2023.

\bibitem{HH03}
Thomas~Y. Hou and Gang Hu.
\newblock A nearly optimal existence result for slightly perturbed 3-{D} vortex
  sheets.
\newblock {\em Communications in Partial Differential Equations},
  28(1-2):155--198, 2003.

\bibitem{HLS97}
Thomas~Y. Hou, John~S. Lowengrub, and Michael~J. Shelley.
\newblock The long-time motion of vortex sheets with surface tension.
\newblock {\em Physics of Fluids}, 9(7):1933--1954, 1997.

\bibitem{Lannes13}
David Lannes.
\newblock A stability criterion for two-fluid interfaces and applications.
\newblock {\em Archive for Rational Mechanics and Analysis}, 208(2):481--567,
  2013.

\bibitem{L02}
Gilles Lebeau.
\newblock Régularité du problème de {K}elvin-{H}elmholtz pour l'équation
  d'{E}uler 2{D}.
\newblock {\em ESAIM Control, Optimisation and Calculus of Variations},
  8:801--825, 2002.

\bibitem{MB02}
Andrew~J. Majda and Andrea~L. Bertozzi.
\newblock {\em Vorticity and incompressible flow}, volume~27 of {\em Cambridge
  Texts in Applied Mathematics}.
\newblock Cambridge University Press, Cambridge, 2002.

\bibitem{MMS2024}
Riccardo Montalto, Federico Murgante, and Stefano Scrobogna.
\newblock Quadratic lifespan for the sublinear $ \alpha $-{SQG} sharp front
  problem.
\newblock {\em arXiv preprint arXiv:2402.06364}, 2024.

\bibitem{PS20}
Bartosz Protas and Takashi Sakajo.
\newblock Rotating equilibria of vortex sheets.
\newblock {\em Physica D}, 403:132286, 2020.

\bibitem{Roulley2023c}
Emeric Roulley.
\newblock Local and global bifurcation of electron states.
\newblock {\em arXiv preprint arXiv:2309.11260}, 2023.

\bibitem{Roulley2023a}
Emeric Roulley.
\newblock Periodic and quasi-periodic {E}uler-{$\alpha$} flows close to
  {R}ankine vortices.
\newblock {\em Dynamics of Partial Differential Equations}, 20(4):311--366,
  2023.

\bibitem{SZ11}
Jalal Shatah and Chongchun Zeng.
\newblock Local well-posedness for fluid interface problems.
\newblock {\em Archive for Rational Mechanics and Analysis}, 199(2):653--705,
  2011.

\bibitem{SS85}
Catherine Sulem and Pierre-Louis Sulem.
\newblock Finite time analyticity for the two- and three-dimensional
  {R}ayleigh-{T}aylor instability.
\newblock {\em Transactions of the American Mathematical Society},
  287(1):127--160, 1985.

\bibitem{SSBF81}
Catherine Sulem, Pierre-Louis Sulem, Claude Bardos, and Uriel Frisch.
\newblock Finite time analyticity for the two- and three-dimensional
  {K}elvin-{H}elmholtz instability.
\newblock {\em Communications in Mathematical Physics}, 80(4):485--516, 1981.

\bibitem{Wu06}
Sijue Wu.
\newblock Mathematical analysis of vortex sheets.
\newblock {\em Communications on Pure and Applied Mathematics},
  59(8):1065--1206, 2006.

\end{thebibliography}
		\bibliographystyle{plain}
	\end{footnotesize}
	\vspace{2cm}
	\noindent FEDERICO MURGANTE: Università Statale di Milano, Via Saldini 50 Milano, 20133 Italy.\\
	E-mail address: federico.murgante@unimi.it\\
	
	\noindent EMERIC ROULLEY: SISSA International School for Advanced Studies, Via Bonomea 265 Trieste, 34136 Italy.\\
	E-mail address: eroulley@sissa.it\\
	
	\noindent STEFANO SCROBOGNA: Department of Mathematics and Geosciences, University of Trieste, via Valerio 12/1 Trieste, 34127 Italy.\\
	E-mail address: stefano.scrobogna@units.it
\end{document}